\documentclass{amsart}

\usepackage{amsmath, amsthm, amssymb, amsfonts, mathtools}
\usepackage{graphicx,epsfig,color}
\usepackage{scrpage2}
\usepackage{exscale}

\theoremstyle{plain}

\newtheorem{theorem}{Theorem}[section]

\newtheorem{corollary}[theorem]{Corollary}

\newtheorem{proposition}[theorem]{Proposition}

\theoremstyle{definition}

\newtheorem{definition}[theorem]{Definition}

\theoremstyle{remark}
\newtheorem{remark}[theorem]{Remark}

\newtheorem{example}[theorem]{Example}

\newtheorem*{ak}{Acknowledgements}

\newcommand{\form}{\alpha}

\DeclareSymbolFont{AMSb}{U}{msb}{m}{n}
\DeclareMathSymbol{\N}{\mathalpha}{AMSb}{"4E}
\DeclareMathSymbol{\R}{\mathalpha}{AMSb}{"52}
\DeclareMathSymbol{\Z}{\mathalpha}{AMSb}{"5A}
\DeclareMathSymbol{\D}{\mathalpha}{AMSb}{"44}
\DeclareMathSymbol{\s}{\mathalpha}{AMSb}{"53}

\newcommand{\sF}{\scriptscriptstyle{F}}
\newcommand{\sC}{\scriptscriptstyle{C}}
\newcommand{\sB}{\scriptscriptstyle{B}}
\newcommand{\sX}{\scriptscriptstyle{X}}
\newcommand{\sM}{\scriptscriptstyle{M}}
\newcommand{\sN}{\scriptscriptstyle{N}}
\newcommand{\sK}{\scriptscriptstyle{K}}
\newcommand{\sZ}{\scriptscriptstyle{Z}}

\newcommand{\sscr}{\scriptscriptstyle}
\newcommand{\mms}{(X,\de_{{\scriptscriptstyle{X}}},\m_{{\scriptscriptstyle{X}}})}

\DeclareMathOperator{\tr}{tr}

\DeclareMathOperator{\vol}{vol}

\DeclareMathOperator{\supp}{supp}

\DeclareMathOperator{\de}{d}
\DeclareMathOperator{\M}{M}

\DeclareMathOperator{\X}{X}
\DeclareMathOperator{\F}{F}
\DeclareMathOperator{\m}{m}
\DeclareMathOperator{\ric}{ric}
\DeclareMathOperator{\diam}{diam}

\DeclareMathOperator{\Ent}{Ent}

\title{Lagrangian calculus for nonsymmetric diffusion operators}
\author{Christian Ketterer}
\email{christian.ketterer@math.uni-freiburg.de}
\begin{document}
\maketitle
\begin{abstract}
We characterize lower bounds for the Bakry-Emery Ricci tensor of nonsymmetric diffusion operators
by convexity 
of entropy on the $L^2$-Wasserstein space, and
define a curvature-dimension condition for general metric measure spaces together with a square integrable $1$-form in the sense of \cite{giglinonsmooth}.
This extends the Lott-Sturm-Villani approach for lower Ricci curvature bounds of metric measure spaces. In generalized smooth context, consequences are new Bishop-Gromov estimates, pre-compactness under measured Gromov-Hausdorff convergence, and a Bonnet-Myers theorem that
generalizes previous results by Kuwada \cite{kuwadamaximaldiameter}. We show that $N$-warped products together with lifted vector fields satisfy the curvature-dimension condition. For smooth Riemannian manifolds we derive an evolution variational inequality and contraction estimates for 
the dual semigroup of nonsymmetric diffusion operators. Another theorem of Kuwada \cite{kuwadaduality, kuwadaspacetime} yields Bakry-Emery gradient estimates.
\end{abstract}

\tableofcontents
\section{Introduction}
\noindent
In this article we present a Lagrangian approach for studying possibly nonsymmetric diffusion operators.
For instance, we consider operators of the form $L=\Delta+\alpha$ where $\Delta$ is the Laplace-Beltrami operator of a compact smooth Riemannian manifold $(M,g_{\sM})$ and $\alpha$ is a smooth $1$-form on $M$.
Then the Bakry-Emery $N$-Ricci tensor associated to $L$ is defined by
\begin{align*}
\ric_{\sM,\alpha}^{\sN}=\ric_{\sM}-\nabla^s\alpha-\textstyle{\frac{1}{N}}\alpha\otimes \alpha
\end{align*}
for $N\in (0,\infty]$
where $2\nabla^s\form(v,w)=\nabla_v\alpha+\nabla_w\form$ denotes the symmetric derivation of $\form$ with respect to the Levi-Civita connection $\nabla$ of $M$. 
Note $\ric_{\sM,\alpha}^{\sN}$ is also meaningful for $N\leq 0$ but we will not consider these cases.

If $\form=-df$, 
$L$ is the diffusion operator of the canonical symmetric Dirichlet form associated to 
the smooth metric measure space $(M,g_{\sM},\m_{\sM})$ with $\m_{\sM}=e^{-f}\vol_{\sM}$,
and $\ric_{\sM,-df}^{\sN}$ is understood as the Ricci curvature of $(M,g_{\sM},\m_{\sM})$.
In celebrated articles by Lott, Sturm and Villani \cite{lottvillani, stugeo1, stugeo2} - built on previous results in \cite{ottovillani, cms, sturmrenesse} - a definition of lower Ricci curvature bounds for 
general metric measure spaces in terms of convexity properties of entropy functionals on 
the $L^2$-Wasserstein space was introduced. In smooth context these definitions are equivalent to lower bounds for the Bakry-Emery tensor provided $\alpha$ is exact.

For diffusion operators where $\alpha$ is not necessarily exact such a geometric picture was missing. 
Though the operator $L$ yields a bilinear form, in general this form is not symmetric and therefore cannot arise as
Dirichlet form of a metric measure space.
Nevertheless, there are numerous results dealing with probabilistic, analytic and geometric properties of $L$ under lower bounds on 
$\ric_{\sM,\alpha}^{\sN}$. The results are very similar to properties that one derives for symmetric operators with lower bounded Ricci curvature, e.g. 
\cite{kuwadamaximaldiameter, kuwadaspacetime, wang}.

In this article we derive a geometric picture associated to the diffusion operator $L=\Delta + \form$ for general $1$-forms $\form$ in the spirit of the work by Lott, Sturm and Villani. 
We characterize lower bounds on $\ric_{\sM,\alpha}^{\sN}$ in terms of convexity for line integrals along $L^2$-Wasserstein geodesics. 
Moreover, for generalized smooth metric measures spaces (Definition \ref{gsmms}) we impose the following definition. For simplicity, in this introduction we assume $N=\infty$. 
We will say $(X,\de_{\sX},\m_{\sX})$ together with a $1$-form $\alpha$ satisfies the curvature-dimension condition $CD(K,\infty)$ if and only if for every pair $\mu_0,\mu_1\in\mathcal{P}^2(\m_{\sX})$ there
 exists an $L^2$-Wasserstein geodesic $\Pi$ such that 
\begin{align*}
\Ent(\mu_t)-\phi_{t}(\Pi)\leq (1-t)\Ent(\mu_0)&+t\left[\Ent(\mu_1)-\phi_{1}(\Pi)\right]
\\
&\hspace{1cm}
-\frac{1}{2}Kt(1-t)KW_2(\mu_0,\mu_1)^2,
\end{align*}
where $\phi_t(\Pi)=\int \int_0^t\form(\dot{\gamma}(s))dsd\Pi(\gamma)$ and $\int_0^t\alpha(\dot{\gamma}(s))ds$ denotes the line integral of $\alpha$ along $\gamma$.
For the case $N<\infty$ corresponding definitions are made in Definition \ref{defA} and Definition \ref{defB}.
In particular, we emphasize that Definition \ref{defB} is also meaningful in the class of general metric measure spaces
together with $L^2$-integrable $1$-forms in the sense of \cite{giglinonsmooth}. 
However, in this article we will only study the generalized smooth case.

We prove several geometric consequences: Generalized Bishop-Gromov estimates, 
pre-compactness under Gromov-Hausdorff convergence, and a generalized Bonnet-Myers Theorem.
The latter generalizes a result of Kuwada in \cite{kuwadamaximaldiameter} - even 
for smooth ingredients. 
Then, we show that the condition $CD(K,N)$ is stable under $N$-warped product constructions. This also includes so-called euclidean $N$-cones and $N$-suspensions.

In the last section we introduce the notion of $EVI_{\sK}$-flows that arise naturally on generalized smooth metric measure spaces together with a $1$-form satisfying
a curvature-dimension condition. More preciely, if $P_t$ is the semigroup associated to the operator $L=\Delta+\alpha$ on a smooth Riemannian manifold
$M$ such that $(M,g_{\sM},\vol_{\sM},\alpha)$ satisfies $CD(K,\infty)$, and if 
$\mathcal{H}_t$ is the dual flow acting on probability measures, then it is an absolutely continuous curve in $L^2$-Wasserstein space and for any probability measure $\mu$, $\mathcal{H}_t\mu$ satisfies the following inequality
\begin{align*}
\frac{1}{2}\frac{d}{ds}W^2_2(\mathcal{H}_s\mu,\nu)+\frac{K}{2}W^2_2(\mathcal{H}_s\mu,\nu)\leq \int_0^1\int \form(\dot{\gamma})d\Pi^s(\gamma)dt + \Ent(\nu)-\Ent(\mathcal{H}_s(\mu))
\end{align*}
provided standard regularity assumption for the corresponding heat kernel (Proposition \ref{propA}).
$\nu$ is any absolutely continuous probability measure, and $\Pi^s$ is the $L^2$-Wasserstein geodesic between $\mathcal{H}_s\mu$ and $\nu$.  
This yields a contraction estimate for $\mathcal{H}_t$ that is again equivalent to the Bakry-Emery conditon for $P_t$ by \cite{kuwadaduality}.
Therefore, we obtain the following theorem.
\begin{theorem}
Let $(M,g_{\sM})$ be a compact smooth Riemannian manifold, and let $\alpha$ be a smooth $1$-form. We denote with $(M,\de_{\sM},\vol_{\sM})$ the corresponding metric measure space, 
and let $P_t$ and $\mathcal{H}_t$ be as in the section 7.
Then, the following statements are equivalent.
\begin{itemize}
\smallskip
 \item[(i)]$\ric_{\sM,\alpha}^{\infty}\geq K$, 
\smallskip
 \item[(ii)]$(M,\de_{\sM},\vol_{\sM},\alpha)$ satisfies the condition $CD(K,\infty)$,
\smallskip
 \item[(iii)]For every $\mu\in\mathcal{P}^2(X)$ $\mathcal{H}_t\mu$ is an $EVI_{K,\infty}$-flow curve starting in $\mu$,
\smallskip
 \item[(iv)]$\mathcal{H}_t$ satisfies the contraction estimate in corollary \ref{contractionxxx},
\smallskip
 \item[(v)]$P_t$ satisfies the condition $BE(K,\infty)$.
\end{itemize}
\end{theorem}
\begin{ak}
This work was partly done while the author was in residence at the Mathematical Sciences Research Institute in Berkeley, California during the Spring 2016 semester, supported
by the National Science Foundation. I want to thank the organizers
of the Differental Geometry Program and MSRI for providing great environment for research. 
\end{ak}

\section{Preliminaries}
\noindent
\paragraph{\textbf{Metric measure spaces}}
Let $(X,\de_{\sX})$ be a complete and separable metric space, and let $\m_{\sX}$ be a locally finite Borel measure. 
We call $(X,\de_{\sX},\m_{\sX})$ a metric measure space.
The case $\m_{\sX}(X)=0$ is excluded.
The space of constant speed geodesics $\gamma:[0,1]\rightarrow X$ is denoted with $\mathcal{G}(X)$, 
and it is equipped with the topology of uniform convergence. $e_t:\gamma\mapsto \gamma(t)$ denotes the evaluation map at time $t$ that is continuous.
\smallskip\\
The $L^2$-Wasserstein space of probability measures with finite second moment is denoted with 
$\mathcal{P}^2(X)$, and $W_2$ is the $L^2$-Wasserstein distance. 
$\mathcal{P}^2_c(X)$ and $\mathcal{P}^2(\m_{\sX})$ denote the subset of compactly supported probability measures and
the family of $\m_{\sX}$-absolutely continuous probability measures, respectively.
A coupling or plan between probability measures $\mu_0$ 
and $\mu_1$ is a probability measure $\pi\in \mathcal{P}(X^2)$ such that $(p_i)_{\star}\pi=\mu_i$ 
where $(p_i)_{i=0,1}$ are the projection maps. A coupling $\pi$ is optimal if 
$
\int_{X^2}\de(x,y)^2d\pi(x,y)=W_2(\mu_0,\mu_1)^2.
$
Optimal couplings exist, and if an optimal coupling $\pi$ is induced by a map $T:Y\rightarrow X$ via 
$(T,\mbox{id}_{\sX})_{\star}\mu_0=\pi$ where $Y$ is a measurable subset of $X$, we say $T$ is an optimal map.
\smallskip\\
A probability measure $\Pi\in \mathcal{P}(\mathcal{G}(X))$ is called an optimal dynamical coupling if $(e_0,e_1)_{\star}\Pi$ is an optimal coupling between its marginal distributions. 
Let $(\mu_t)_{t\in[0,1]}$ be an $L^2$-Wasserstein geodesic in $\mathcal{P}^2(X)$. We say an optimal dynamical coupling $\Pi$ is a lift of $\mu_t$ if $(e_t)_{\star}\Pi=\mu_t$ for every $t\in [0,1]$.
If $\Pi$ is the lift of an $L^2$-Wasserstein geodesic $\mu_t$, we call $\Pi$ itself an $L^2$-Wasserstein geodesic.
We say $\Pi$ has bounded compression if there exists a constant $C:=C(\Pi)$ such
that $(e_t)_{\star}\Pi\leq C(\Pi)\m_{\sX}$ for every $t\in [0,1]$. 
\smallskip\\
We say that a metric measure space $\mms$ is \textit{essentially non-branching} if for any optimal dynamical coupling $\Pi$ there exists $A\subset \mathcal{G}(X)$ such that $\Pi(A)=1$ and for all $\gamma, \gamma'\in A$ we have that
$
\gamma(t)=\gamma'(t) \mbox{ for all }t\in [0,\epsilon] \mbox{ and for some }\epsilon>0\ \mbox{implies} \ \gamma=\gamma'.
$
For instance a metric measure space satisfying a Riemannian curvature-dimension condition $RCD(K,N)$ in the sense of \cite{erbarkuwadasturm, giglistructure} is essentially non-branching \cite{rajalasturm}.
\smallskip\\
If we assume that for $\m_{\sX}\otimes\m_{\sX}$-almost every pair $(x,y)\in X^2$ there exists a unique geodesic $\gamma_{x,y}\in\mathcal{G}(X)$ between $x$ and $y$
then by measurable selection there exists a measurable map $\Psi:X^2\rightarrow \mathcal{G}(X)$ with $\Psi(x,y)=\gamma_{x,y}$. 
For $\mu\in\mathcal{P}^2(X)$ and $A\subset X$ Borel with $\m_{\sX}(A)>0$ we set $\mathcal{M}_{\mu,A}=\Psi_{\star}(\mu\otimes\m_{\sX}|_{A})$, and for $x_0\in X$ we set $\mathcal{M}_{\delta_{x_0},A}=\mathcal{M}_{{x_0},A}$.
In this case $\Pi_{x_0,A}:=\m_{\sX}(A)^{-1}\mathcal{M}_{x_0,A}$ is the unique optimal dynamical plan between $\delta_{x_0}$ and $\m_{\sX}(A)^{-1}\m_{\sX}|_{A}$.
Again, the family of $RCD$-spaces is a class that satisfies this property \cite{giglirajalasturm}.
\smallskip
\begin{definition}[Generalized smooth metric measure spaces]\label{gsmms} We say
$(X,\de_{\sX},\m_{\sX})$ is a \textit{generalized smooth metric measure space}
if there exists an open smooth manifold $M_{\sX}=M$, and a Riemannian metric $g_{\sM}$ on $M$ with induced distance function $\de_{\sM}$ 
such that the metric completion of the metric space $(M,\de_{\sM})$ is isometric to $(X,\de_{\sX})$, 
and for any optimal dynamical plan $\Pi\in\mathcal{P}(\mathcal{G}(X))$ such that $(e_t)_{\star}\Pi=\mu_t$ is a geodesic and $\mu_0\in\mathcal{P}^2(\m_{\sX})$
we have that 
\begin{align}\label{transportcondition}
\Pi(S_{{\sscr{\Pi}}})=0 \ \mbox{ where }\ S_{{\sscr \Pi}}:=\left\{\gamma\in\mathcal{G}(X):\exists t\in (0,1)\mbox{ s.t. }\gamma(t)\in X\backslash M\right\}.
\end{align}
In particular, $\mu_t(e_t(S_{{\sscr \Pi}}))=0$, and if we choose the constant geodesic $\Pi$ with $(e_t)_{\star}\Pi=\m_{\sX}(K)^{-1}\m_{\sX}|_{K}$ for all $t\in [0,1]$ 
where $K\subset X$ is any measurable set of finite $\m_{\sX}$-measure, 
one gets that $K\cap X\backslash M$ is of $\m_{\sX}$-measure $0$. We call $M$ the set of regular points in $X$. 
\end{definition}
\begin{remark}
The condition (\ref{transportcondition}) yields that $\mms$ is essentially non-branching and that 
for $\m_{\sX}\otimes\m_{\sX}$-almost every pair $(x,y)\in X^2$ there exists a unique geodesic $\gamma_{x,y}\in\mathcal{G}(X)$ between $x$ and $y$. Moreover, 
for each pair $\mu_0,\mu_1\in\mathcal{P}^2_c(\m_{\sX})$ there is a unique dynamcial optimal coupling $\Pi$ such that 
$\Pi(\mathcal{G}(M))=1$ where $\mathcal{G}(M)$ is the space of geodesic in $M$, $(e_t)_{\star}\Pi\in \mathcal{P}^2(\m_{\sX})$,
and $\Pi$ is induced by a map. To see this note that $\mu_i(M)=1, i=0,1$ and transport geodesics are contained in $M$. Then, since one can choose an exhaustion of $M$ by 
compact sets, we can assume that $\mu_0$ and $\mu_1$ are compactly supported in $M$. Then, the claim follows from statements in \cite{cms} and since geodesics are unique.
\end{remark}
\noindent
Examples of generalized smooth metric measure spaces in the sense of the previous definition are Riemannian manifolds with boundary that are geodesically convex, cones, suspensions \cite{bastco} and
warped products \cite{ketterer}. Moreover, in an upcoming paper of the author with Ilaria Mondello, it will be shown that stratified spaces are generalized smooth provided certain assumptions on tangent cones at singular points.
This result will also show that orbifolds are generalized smooth.
\smallskip
\paragraph{\textbf{$1$-forms and vector fields}} Assume $\mms$ is a generalized smooth metric measure space.
A $1$-form $\alpha$ is a measurable map $\alpha:X\rightarrow T^*M$ with $\alpha(x)=T_x^*M$.
We say $\alpha\in L^p_{loc}(\m_{\sX},TM^*)$ for $p\in[1,\infty)$ if 
$
\left\|\alpha\right\|_{L^{p}(\m_{\sX},K)}^p=\int_{K\cap M} |\alpha|^p_{\sM}d\m_{\sX}
$ is finite
where $|\alpha|_{\sM}^2={g^*_{\sM}(\alpha,\alpha)}$ and $K\subset X$ compact. 
\smallskip\\
Similar, we can consider measurable and $L^p$-integrable vector fields on $X$.
Note that a vector field $Z$ on $M$ yields a $1$-form $\alpha$ via $\alpha=\langle Z,\cdot\rangle$. In the context of generalized smooth metric measures space 
this is the natural isomorphism between vectorfields and $1$-forms, and we will often switch between these viewpoints.
If $\Pi$ is an optimal dynamical coupling with bounded compression, the line integral
$$\phi_t(\gamma):=\phi^{\alpha}_t(\gamma):=\int_0^t\alpha(\dot{\gamma})(\tau)d\tau$$ 
exists $\Pi$-almost surely, and it does not depend on the parametrization of $\gamma$ up to changes of orientation. 
Moreover, for any $L^2$-Wasserstein geodesic $\Pi$ with bounded compression, we set $\phi_t(\Pi)=\int \phi_t(\gamma) d\Pi(\gamma)$.
\smallskip
\paragraph{\textbf{The case of arbirtrary metric measure spaces}}
Let $\phi:X\rightarrow \mathbb{R}\cup \left\{\pm\infty\right\}$ be any function. The Hopf-Lax semigroup $Q_t:X\rightarrow \mathbb{R}\cup \left\{-\infty\right\}$ is defined by
$
Q_t\phi(x)=\inf_y\frac{1}{2t}\de(x,y)^2+\phi(y).
$
$Q_1(-\phi)$ is the $c$-transform of $\phi$. We say a function $\phi$ is $c$-concave if there exist $v:X\rightarrow \mathbb{R}\cup \left\{\pm\infty\right\}$ such that $\phi=Q_1v$.
If $X$ is compact, by Kantorovich duality for any pair $\mu_0,\mu_1\in\mathcal{P}^2(\m_{\sX})$ with bounded densities there exist a Lipschitz function $\phi$ such that 
\begin{align}\label{phi}
W_2(\mu_0,\mu_t)^2=\int Q_t\phi d\mu_1-\int\phi d\mu_0=\int_0^t\int\frac{d}{ds}Q_s\phi\big|_{s=t} \rho_t d\m_{\sX} dt 
\end{align}
for any geodesic $t\mapsto \mu_t$ with bounded compression. For instance, see \cite{giglihancontinuity}.
\smallskip\\
If we follow the approach of Gigli in \cite{giglinonsmooth}, there is also a well-defined notion of $L^p(\m_{\sX})$-integrable $1$-form $\form$ for general metric measure spaces $\mms$, and one can define
the dual coupling $\form(\nabla f):X\rightarrow \mathbb{R}$ as measurable function on $X$ where $f$ is a Sobolev function. Note that in this context $\nabla f$ does not necessarily exist.
The notion of line integral along a geodesic is more subtle, but if we consider a $L^2$-Wasserstein geodesic $\Pi\in\mathcal{P}(\mathcal{G}(X))$ that has bounded compression, then we can define
$
\bar{\phi}_t(\Pi)=-\int_0^t\int \form(\nabla Q_s\phi)\rho_s d\m_{\sX}ds
$
where $Q_t\phi$ is a Kantorovich potential for $\Pi$, and $\rho_t$ is the density of $(e_t)_{\star}\Pi$ .
Since $\rho_t$ and $\nabla Q_t\phi$ are bounded, $\bar{\phi}_t(\Pi)$ is well-defined if $|\form|$ is $\m_{\sX}$-integrable. Note that in a smooth context $\dot{\gamma}(t)=-\nabla Q_t\phi|_{\gamma(t)}$, and 
therefore in smooth context we have \begin{align*}
\bar{\phi}_t(\Pi)&=\int_0^t\int \form(\nabla Q_s\phi|_x)\rho_s(x) d\m_{\sX}(x)ds= \int_0^t \int\form(\nabla Q_s\phi|_{\gamma(s)})d\Pi(\gamma)ds\\
&= \int \int_0^t\form(\nabla Q_s\phi|_{\gamma(s)})dsd\Pi(\gamma)= \int\phi_t(\gamma)d\Pi(\gamma)=\phi_t(\Pi).
\end{align*}
In the following we just write $\alpha_t(\Pi)$ for $\bar{\alpha}_t(\Pi)$.
Also note, that in general there is no identification between $1$-forms and vectorfields. 
\smallskip
\paragraph{\textbf{Entropy functionals}}
For $\mu\in\mathcal{P}^2(X)$ we define the Boltzmann-Shanon entropy by
\begin{align*}
\Ent(\mu):=\int \log\rho d\mu\ \mbox{ if }\ \mu=\rho\m_{\sX}\mbox{ and } (\rho\log\rho)_+ \mbox{ is }\m_{\sX}\mbox{-integrable},
\end{align*}
and $\Ent(\mu)=+\infty$ otherwise. 
Given a number $N\geq 1$, we define the $N$-R\'eny entropy functional $S_{\sN}:\mathcal{P}^2(X)\rightarrow (-\infty,0]$ with respect to $\m_{\sX}$ by
\begin{align*}
S_{\sN}(\mu):=-\int\rho^{1-\frac{1}{N}}(x)d\m_{\sX}
\end{align*}
where $\rho$ denotes the density of the absolutely continuous part in the Lebesgue decomposition of $\mu$. 
In the case $N=1$ the $1$-R\'eny entropy is $-\m_{\sX}(\supp\rho)$. 
If $\m_{\sX}$ is finite, then 
\begin{align*}
-\m_{\sX}(X)^{\frac{1}{N}}\leq S_{\sN}(\cdot)\leq 0
\end{align*} and 
$\Ent(\mu)=\lim_{N\rightarrow \infty}N(1+S_{\sN}(\mu))$.
Moreover, if $\m_{\sX}$ is finite and $N>1$, then $S_{\sN}$ is lower semi-continuous. If $\m_{\sX}$ is $\sigma$-finite one has to assume an exponential growth condition \cite{agmr} to guarantee lower semi continuity.
\medskip\\
If there is a $1$-form $\form$, we also define $S^{\form}_{\sN,t}:\mathcal{P}(\mathcal{G}(X))\rightarrow (-\infty,0]$ by
\begin{align*}
S^{\form}_{\sN,t}(\Pi):=-\int\rho_t^{-\frac{1}{N}}(\gamma_t)e^{\frac{1}{N}\phi_t(\gamma)}d\Pi(\gamma)
\ \ \mbox{ if $(e_t)_{\star}\Pi=\rho_t\m_{\sX}$}
\end{align*}
and $0$ otherwise. If $\form=0$, then $S^{\form}_{\sN,t}(\Pi)=S_{\sN}((e_t)_{\star}\Pi)$.
\medskip
\paragraph{\textbf{Distortion coefficients}}
For two numbers $K\in\mathbb{R}$ and $N\geq 1$ we define 
\begin{align*}
(t,\theta)\in [0,1]\times (0,\infty)\mapsto\sigma_{\sK,\sN}^{(t)}(\theta)=
\begin{cases} 
\frac{\sin_{\sK/\sN}(t\theta)}{\sin_{\sK/\sN}(\theta)} \ & \ \mbox{ if } \sin_{\sK/N}(x)>0 \mbox{ for }x\in (0,\theta],\\
\infty \ & \ \mbox{ otherwise}.
                               \end{cases}
\end{align*}
$\sin_{\sK/\sN}$ is the solution of the initial value problem 
\begin{align*}
u''+\textstyle{\frac{K}{N}}u=0, \ \ u(0)=0 \ \& \ u'(0)=1.
\end{align*}
The modified distortion coefficients for number $K\in\mathbb{R}$ and $N>1$ are given by 
\begin{align*}
(t,\theta)\in [0,1]\times (0,\infty)\mapsto \tau_{\sK,\sN}^{(t)}(\theta)=\begin{cases} \theta\cdot\infty \ & \ \mbox{ if }K>0 \mbox{ and }N=1,\\                                                                          
t^{\frac{1}{N}}\left[\sigma_{\sK/(\sN-1)}^{(t)}(\theta)\right]^{1-\frac{1}{N}} \ & \ \mbox{ otherwise}.
                                                                         \end{cases}
\end{align*}
\section{Curvature-dimension condition for nonsymmetric diffusions}
\begin{definition}\label{defA}
Let $(X,\de_{\sX},\m_{\sX})$ be a generalized smooth metric measure space, and let $\alpha$ be an $L^2$-integrable $1$-form. We say $(X,\de_{\sX},\m_{\sX},\alpha)$ satisfies 
the \textit{curvature-dimension condition $CD(K,N)$} for $K\in\mathbb{R}$ and $N\geq 1$ if and only if
for
each pair $\mu_0, \mu_1\in\mathcal{P}_c^2(X,\m_{\sX})$ there exists a dynamical optimal plan $\Pi$ with
\begin{align*}
S_{\sN,t}^{\form}(\Pi)
\leq  -\!\!\int\left[\tau_{\sK,\sN}^{(1-t)}(|\dot{\gamma}|)\rho_0(\gamma_0)^{-\frac{1}{N}}+
\tau_{\sK,\sN}^{(1-t)}(|\dot{\gamma}|)e^{\frac{1}{N}\phi_1(\gamma)}\rho_1(\gamma_1)^{-\frac{1}{N}}\right]d\Pi(\gamma)
\end{align*}
where $\phi_t(\gamma)=\int_0^t\form(\dot{\gamma})(\tau)d\tau$. We call any such $1$-form $\form$ admissible.
\smallskip\\
If we replace $\tau_{\sK,\sN}^{(t)}(\theta)$ by $\sigma_{\sK,\sN}^{(t)}(\theta)$ in the previous definition 
we say $(X,\de_{\sX},\m_{\sX},\alpha)$ satisfies the \textit{reduced curvature-dimension condition} $CD^*(K,N)$.
\end{definition}
\begin{definition}\label{defB}
Let $(X,\de_{\sX},\m_{\sX})$ be a metric measure space, and let $\alpha$ be an $L^2$-integrable $1$-form in the sense of \cite{giglinonsmooth}.
\smallskip\\
We say $(X,\de_{\sX},\m_{\sX},\alpha)$ satisfies the \textit{curvature-dimension condition} $CD(K,\infty)$ if and only if for
each pair $\mu_0, \mu_1\in\mathcal{P}^2(\m_{\sX})$ with bounded densities there exists a geodesic $\Pi$ with bounded compression and a potential $\phi$ as in (\ref{phi}) such that
\begin{align*}
\Ent(\mu_t)-\phi_{t}(\Pi)\leq (1-t)\Ent(\mu_0)&+t\left[\Ent(\mu_1)-\phi_{1}(\Pi)\right]
\\
&\hspace{1cm}
-\frac{1}{2}Kt(1-t)KW_2(\mu_0,\mu_1)^2,
\end{align*}
where $\phi_{t}(\Pi)=\int_0^t\int \alpha(\nabla Q_t\phi)\rho_t d\m_{\sX}dt$ and $\mu_t=(e_t)_{\star}\Pi$. 
Equivalently, the map $t\mapsto 
\Ent(\mu_t)-\phi_{t}(\Pi)$ is $K$-convex.
\smallskip\\
$(X,\de_{\sX},\m_{\sX},\alpha)$ satisfies the \textit{entropic curvature-dimension condition} $CD^e(K,N)$ if and only if for
each pair $\mu_0, \mu_1\in\mathcal{P}^2(\m_{\sX})$ with bounded densities there exists a geodesic $\Pi$ with bounded compression and a potential $\phi$ as in (\ref{phi}) such that
\begin{align*}
U_{\sN}(\mu_t)e^{\frac{1}{N}\phi_{t}(\Pi)}\leq 
\sigma_{\sK/\sN}^{(1-t)}(W_2(\mu_0,\mu_1))U_{\sN}(\mu_0)+
\sigma_{\sK/\sN}^{(1-t)}(W_2(\mu_0,\mu_1))e^{\frac{1}{N}\phi_{1}(\Pi)}U_{\sN}(\mu_1)
\end{align*}
where $U_{\sN}(\mu)=e^{-\frac{1}{N}\Ent(\mu_t)}$ and $\mu_t=(e_t)_{\star}\Pi$. That is the map $t\mapsto 
\Ent(\mu_t)-\phi_{t}(\Pi)$ is $(K,N)$-convex in the sense of \cite{erbarkuwadasturm}.
\end{definition}
\begin{remark}
If we can choose $\form=0$ as an admissible $1$-form, the previous definitions become the ones from \cite{lottvillani, stugeo1,stugeo2, bast, erbarkuwadasturm}.
\end{remark}
\begin{remark} It is easy to prove that 
\begin{itemize}
\smallskip
\item[(i)] $CD(K,N)\ \Longrightarrow \ CD^*(K,N)$, 
\smallskip
\item[(ii)] $CD^*(K,N), CD^e(K,N)\ \Longrightarrow \ CD^*(K',N'), CD^e(K',N')$ \\
for $K'\leq K$ and $N'\geq N$, 
\smallskip
\item[(iii)] 
If $\m_{\sX}$ is finite, then $CD^*(K,N), CD^e(K,N) \Longrightarrow \ CD(K,\infty)$. 
\end{itemize} For instance, compare with similar statements in \cite{stugeo2, erbarkuwadasturm}.
\end{remark}
\noindent
\textit{Definition \ref{defB} makes sense for any possibly non-smooth metric measure space.
But for simplicity, for the rest of the article we always assume that $\mms$ is a generalized smooth metric measure space.
Some of the statements that we prove for generalized smooth metric measure spaces extend to arbitrary metric measure spaces but in general not without additional assumptions.
}
\smallskip\\
Let $\mms$ and $(X',\de_{\sX'},\m_{\sX'})$ be generalized smooth metric measure spaces.
A map $I:\supp\m_{\sX'}\rightarrow X$ is a smooth metric measure space isomorphism
if $I$ is a metric measure space isomorphism, and if $I$ is a diffeomophism between the subsets of regular points $M'$ and $M$.
\begin{proposition}
Let $(X,\de_{\sX},\m_{\sX})=\X$ be a generalized smooth metric measure space, and let $\alpha$ be an $L^2$-intergrable $1$-form. Assume $(\X,\alpha)$
satisfies the condition $CD(K,N)$. Then the following properties hold.
\begin{itemize}
\smallskip
 \item[(i)] For $\eta,\beta >0$ define the generalized smooth metric measures space \linebreak $(X,\eta\de_{\sX},\beta\m_{\sX})=:\X'$. Then $(\X',\alpha)$ satisfies $CD(\eta^{-2}K,N)$.
 \smallskip
 \item[(ii)] For a convex subset $X'\subset X$ define the generalized smooth metric measure space $(X',\de|_{\sX'\times \sX'},\m|_{\sX})=:\X'$. 
 Then $(\X',\alpha|_{X'})$ satisfies $CD(K,N)$.
 \smallskip
 \item[(iii)]
 Let $\X'$ be a generalized smooth metric measure space, and let $I:X'\rightarrow X$ be a smooth metric measure space isomorphism. Then $(\X',I^{\star}\alpha)$ satisfies the condition $CD(K,N)$.
\end{itemize}
\end{proposition}
\begin{proof}
We check (iii). We define $\form'= I^{\star}\alpha$ on $I^{-1}(M)$. 
If $\gamma$ is a geodesic in $X$, then $I^{-1}\circ\gamma=\gamma'$ is a geodesic in $X'$. The line integral of $\form'$ along $\gamma'$ is
\begin{align*}
\int_0^1 \form'(\dot{\gamma}') dt = \int_0^1 I^{\star}\form ( DI^{-1}|_{\gamma(t)}\dot{\gamma})dt=\int_0^1 \form(\dot{\gamma})dt.
\end{align*}
Then, the statement follows like similar results for metric measure spaces that satisfy a curvature-dimension condition (for instance see \cite{stugeo2}).
\end{proof}

\begin{theorem}\label{nonbranching}
Let $(X,\de_{\sX},\m_{\sX})=\X$ be a generalized smooth metric measure space, $\alpha$ an $L^2$-integrable $1$-form, and $K\in\mathbb{R}$ and $N> 0$. 
Then the following statements are equivalent:
\begin{itemize}
 \item[(i)] $(\X,\form)$ satisfies $CD^*(K,N)$.
 \smallskip
 \item[(ii)] For each pair $\mu_0,\mu_1\in\mathcal{P}_c^2(\m_{\sX})$ there exists an optimal dynamical plan $\Pi$ with $(e_t)_{\star}\Pi=\mu_t\in\mathcal{P}^2(\m_{\sX})$ such that
\begin{align}\label{something1}
\left[\rho_t(\gamma_t)e^{-\phi_t(\gamma)}\right]^{-\frac{1}{N}}
\geq \sigma_{\sK,\sN}^{(1-t)}(|\dot{\gamma}|)\rho_0(\gamma_0)^{-\frac{1}{N}}+\sigma_{\sK,\sN}^{(1-t)}(|\dot{\gamma}|)\left[e^{-\phi_1(\gamma)}\rho_1(\gamma_1)\right]^{-\frac{1}{N}}
\end{align}
for $t\in[0,1]$ and $\Pi$-a.e. $\gamma\in\mathcal{G}(X)$. $\rho_t$ is the density of $\mu_t$ w.r.t. $\m_{\sX}$.
\smallskip
 \item[(iii)] $(\X,\alpha)$ satisfies $CD^e(K,N)$.
\end{itemize}
Moreover, the condition $CD(K,N)$ is equivalent with (ii) if the coefficients $\sigma_{\sK,\sN}^{\sscr{(t)}}(\theta)$ are replaced by the coefficients $\tau_{\sK,\sN}^{\sscr{(t)}}(\theta)$.
\end{theorem}
\begin{proof} First, we observe that in the context of generalized smooth metric measure spaces up to a set measure zero optimal couplings between $\m_{\sX}$-absolutly continuous measures $\mu_0,\mu_1\in\mathcal{P}_c(\m_{\sX})$ are unique
(also compare with the remark after Definition \ref{gsmms}).\smallskip\\
``(i)$\Rightarrow$(ii)'': 
Let $\mu_0,\mu_1\in\mathcal{P}_2(\m_{\sX})$ be with bounded support, and let $\pi\in\mathcal{P}(\mathcal{G}(X))$ be the optimal coupling 
between $\mu_0$ and $\mu_1$. Let $\left\{M_n\right\}_{n\in\mathbb{N}}$ be an $\cap$-stable generator of the Borel $\sigma$-field of $(X,\de_{\sX})$ .
For each $n$ we define a disjoint covering of $X$ of $2^n$ sets by $L_I=\bigcap_{i\in I} M_i \cap \bigcap_{i\in I^c}^{} M_i^c $ where $I\subset \left\{1,\dots,n\right\}$ and $I^c=\left\{1, \dots,n\right\}\backslash I$.
\smallskip\\
We define $B^{\sscr{I},\sscr{J}}=L_{\sscr{I}}\times L_{\sscr{J}}$ and set 
$\pi^{\sscr{I},\sscr{J}}:=\alpha_{\sscr{I},\sscr{J}}^{-1}\pi|_{B^{\sscr{I},\sscr{J}}}$ if $\alpha_{\sscr{I},\sscr{J}}:={\pi(B^{\sscr{I},\sscr{J}})}>0$. 
Then we consider the marginal measures $\mu^{\sscr{I},\sscr{J}}_0=(e_0)_{\star}\pi^{\sscr{I},\sscr{J}}$ and $\mu^{\sscr{I},\sscr{J}}_1=(e_1)_{\star}\pi^{\sscr{I},\sscr{J}}$ that are $\m_{\sX}$-absolutely continuous, and 
$\pi^{\sscr{I},\sscr{J}}$ is the unique optimal coupling. 
Since geodesics are $\m_{\sX}\otimes\m_{\sX}$-almost surely unique, the dynamical optimal plan $\Pi^{\sscr{I},\sscr{J}}=\Psi_{\star}\pi^{\sscr{I},\sscr{J}}$ is the unique optimal dynamical coupling between its endpoints
where $\Psi(x,y)=\gamma_{x,y}\in\mathcal{G}(X)$. Therefore, $\Pi^{\sscr{I},\sscr{J}}$ satisfies the $CD^*$-inequality for every $I$ and $J$. 
In particular, $(e_t)_{\star}\Pi^{\sscr{I},\sscr{J}}=\rho_t^{\sscr{I},\sscr{J}}\m_{\sX}$ is $\m_{\sX}$-absolutely continuous.
Then, we define a dynamical coupling between $\mu_0$ and $\mu_1$ by 
$
\Pi^n:=\sum_{\sscr{I},\sscr{J}\subset\left\{1,\dots,n\right\}}\alpha_{\sscr{I},\sscr{J}}\Pi^{\sscr{I},\sscr{J}}.
$
$\Pi^n$ is optimal since 
\begin{align*}
\pi^n:=(e_0,e_1)_{\star}\Pi_n=\sum_{\sscr{I},\sscr{J}\subset\left\{1,\dots,n\right\}}\alpha_{\sscr{I},\sscr{J}}(e_0,e_1)_{\star}\Pi^{\sscr{I},\sscr{J}}=\sum_{\sscr{I},\sscr{J}\subset\left\{1,\dots,n\right\}}\alpha_{\sscr{I},\sscr{J}}\pi^{\sscr{I},\sscr{J}}=\pi
\end{align*}
is an optimal coupling. Therefore, we can apply Lemma 3.11 in \cite{erbarkuwadasturm}:
Since the measures $\mu_0^{\sscr{I},\sscr{J}}$ for ${\textstyle{I},{J}\subset\left\{1,\dots, 2^n\right\}}$ are mutually singular, 
$\mu_t^{\sscr{I},\sscr{J}}=\rho_t^{\sscr{I},\sscr{J}}d\m_{\sX}$ are mutually singular as well.
\smallskip\\
Now, for $t\in (0,1)$ we consider the measure $\mu_t^n=(e_t)_{\star}\Pi^n$. Since it decomposes into mutually singular, absolutely continuous measures $\mu_t^{\sscr{I},\sscr{J}}$ with densities $\rho_t^{\sscr{I},\sscr{J}}$, 
$\mu_t^n$ is absolutely continuous as well, and by mutual singularity of the measure $\mu_t^{\sscr{I},\sscr{J}}$ its density is $\rho^n_t=\sum\alpha_{\sscr{I},\sscr{J}}\rho_t^{{\sscr{I},\sscr{J}}}$.
Again, since geodesics are $\m_{\sX}\otimes\m_{\sX}$-almost surely unique we have that $\Pi:=\Psi_{\star}\pi=\Psi_{\star}\pi^n=\Pi^n$, and $\rho_t\m_{\sX}=(e_t)_{\star}\Pi=(e_t)_{\star}\Pi^n=\rho^n_t\m_{\sX}$ for every $n\in \mathbb{N}$.
From the $CD^*$-inequality for $\Pi^{\sscr{I},\sscr{J}}$ we have
\begin{align*}
&\int_{L_i\times L_j}\rho_t^{-\frac{1}{N}}(\gamma_{x,y}(t))e^{-\frac{1}{N}\phi_t(\gamma_{x,y})}d\pi(x,y)\\
&=\alpha_{\sscr{I},\sscr{J}}^{1-\frac{1}{N}}\!\!\int(\rho_t^{\sscr{I},\sscr{J}})^{-\frac{1}{N}}(\gamma_{x,y}(t))e^{-\frac{1}{N}\phi_t(\gamma_{x,y})}d\pi^{\sscr{I},\sscr{J}}(x,y)\\
&\geq \alpha_{\sscr{I},\sscr{J}}^{1-\frac{1}{N}}\!\!\int\sigma_{\sK,\sN}^{\sscr{(1-t)}}( |\dot{\gamma}_{x,y}|)(\rho_0^{\sscr{I},\sscr{J}})^{\frac{-1}{N}}(x)+\sigma_{\sK,\sN}^{\sscr{(t)}}( |\dot{\gamma}_{x,y}|)(\rho_1^{\sscr{I},\sscr{J}})^{\frac{-1}{N}}(y)e^{\frac{-1}{N}\phi_1(\gamma_{x,y})}d\pi^{\sscr{I},\sscr{J}}(x,y)\\
&= \int_{L_i\times L_j}\sigma_{\sK,\sN}^{\sscr{(1-t)}}( |\dot{\gamma}_{x,y}|)\rho_0^{-\frac{1}{N}}(x)+\sigma_{\sK,\sN}^{\sscr{(t)}}( |\dot{\gamma}_{x,y}|)\rho_1^{-\frac{1}{N}}(y)e^{-\frac{1}{N}\phi_1(\gamma_{x,y})}d\pi(x,y).
\end{align*}
This holds for every $L_i$ and $L_j$. Since $L_i$ and $L_j$ are mutually disjoint, by summing up the previous inequality holds for $M_i$ and $M_j$ as well. Since the family $\left\{M_j\right\}$ is a generator for the $\sigma$-field, 
we have for $\pi$-almost every $(x,y)\in X\times X$
\begin{align*}\rho_t^{-\frac{1}{N}}(\gamma_{x,y}(t))e^{-\frac{1}{N}\phi_t(\gamma_{x,y})}\geq\sigma_{\sK,\sN}^{\sscr{(1-t)}}( |\dot{\gamma}_{x,y}|)\rho_0^{-\frac{1}{N}}(x)+\sigma_{\sK,\sN}^{\sscr{(t)}}( |\dot{\gamma}_{x,y}|)\rho_1^{-\frac{1}{N}}(y)e^{-\frac{1}{N}\phi_1(\gamma_{x,y})}.
\end{align*}
And since $\Pi=\Psi_{\star}\pi$, this is the claim.
\smallskip\\
For the following recall that in the context of generalized smooth metric measure spaces $$\int_0^1\int\alpha(\nabla Q_t\phi)\rho_t dt=\int \int_0^1\alpha(\dot{\gamma})dt d\Pi$$ where $\Pi$ is an $L^2$-Wasserstein geodesic, $\phi$
is an Kantorovich potential, and
$(e_t)_{\star}\Pi=\rho_t\m_{\sX}$ for every $t\in [0,1]$.
\smallskip\\
``(ii)$\Rightarrow$(iii)'': 
Recall from \cite{erbarkuwadasturm, ketterer5} that $(x,y,\theta)\mapsto G(x,y,\theta)=\log(\sigma_{\sK,\sN}^{\sscr{(t)}}(\theta)e^x+\sigma_{\sK,\sN}^{\sscr{(t)}}(\theta)e^y)$ is convex. Then, 
apply $\log$ to (\ref{something1}) and use Jensen's inequality on the right hand side to obtain the condition $CD^e(K,N)$. 
\smallskip\\
``(iii)$\Rightarrow$(ii)'': This works like in ``(i)$\Rightarrow$(ii)'' where one has to use the convexity of $(x,y,\theta)\mapsto G(x,y,\theta)$ again. See also \cite{erbarkuwadasturm,ketterer5}.
\smallskip\\
``(ii)$\Rightarrow$(i)'': Integrate (\ref{something1}) w.r.t. the optimal dynamical plan $\Pi$.
\smallskip\\
The proof of the equivalence for the condition $CD(K,N)$ is similar.
\end{proof}
\begin{remark}
If we consider $\alpha$ such that $\alpha=-df$ on $M$ for a smooth function $ f:M\rightarrow \mathbb{R}$,
then $$\phi_t(\gamma)=\int_0^t\alpha(\dot{\gamma})dt=f(\gamma(t))-f(\gamma(0))\ \ \& \ \ \phi_1(\gamma)=f(\gamma(1))-f(\gamma(0)),$$ and we can reformulate (\ref{something1}) as 
\begin{align*}
\left[\rho_t(\gamma_t)e^{f(\gamma_t)}\right]^{-\frac{1}{N}}
\geq \tau_{\sK,\sN}^{(1-t)}(|\dot{\gamma}|)\left[\rho_0(\gamma_0)e^{f(\gamma_0)}\right]^{-\frac{1}{N}}+\tau_{\sK,\sN}^{(1-t)}(|\dot{\gamma}|)\left[\rho_1(\gamma_1)e^{f(\gamma_1)}\right]^{-\frac{1}{N}}
\end{align*}
for $\Pi$-a.e. $\gamma$. That is the {condition} $CD(K,N)$ in the sense of \cite{stugeo2} for the metric measure space $(M,\de_{\sX},e^{-f}\vol_{\sX})$.
\end{remark}
\section{The Riemannian manifold situation}
\noindent
In this section we consider a vector field $Z$ rather than a $1$-form $\form$. 
Recall that for smooth metric measure spaces we always can identifiy $Z$ with a $1$-form $\alpha$.
\begin{definition}
Let $(X,\de_{\sX},\m_{\sX})$ be a smooth metric measure space with $(X,\de_{\sX})\simeq(M,\de_{\sM})$ and $\m_{\sX}=\vol_{\sM}$. Let $\nabla$ be the Levi-Civita connection of $g_{\sM}$ and let 
$Z\in L^2(TM)$ be a smooth vector field. We define the 
Bakry-Emery $N$-Ricci tensor for $N\in (n,\infty]$ by
\begin{align*}
\ric_{\sM,\sZ}^{\sN}=\ric_{\sM}-\nabla^sZ-\frac{1}{N-n}Z\otimes Z
\end{align*}
where $\nabla^sZ(v,w)=\frac{1}{2}\left(\langle \nabla_vZ,w\rangle+\langle v,\nabla_wZ\rangle\right)$ and $n=\dim_{\sM}$.
For $N=n$ we define
$$\ric_{\sM,Z}^{\sN}(v):=
\begin{cases}
\ric_{\sM}(v)
-\nabla^{s} Z(v) \ & \ \left\langle Z,v\right\rangle_F=0\\
-\infty \ &\ \mbox{ otherwise}.
\end{cases}$$
For $1\leq N<n$ we define $\ric_{\sM,Z}^{\sN}(v):=-\infty$ for all $v\neq 0$ and $0$ otherwise.
\end{definition}
\begin{theorem}\label{maintheorem}
Let $(X,\de_{\sX},\m_{\sX})$ be a smooth metric measure space
with $(X,\de_{\sX})\simeq(M,\de_{\sM})$ and $\m_{\sX}=\vol_{\sM}$, 
$K\in\mathbb{R}$ and $N\in [1,\infty]$. Let $Z$ be an $L^2$-integrable smooth vector field. Then
$(X,\de_{\sX},\m_{\sX},Z)$ satisfies the condition $CD(K,N)$ if and only if 
\begin{align}\label{assumption}
\ric_{\sM,\sZ}^{\sN}(v)\geq K|v|^2.
\end{align}
Moreover, if $N$ is finite, then $CD(K,N)$ and $CD^*(K,N)$ are equivalent.
\end{theorem}
\begin{proof}\textbf{1.} Assume $\ric^{\sN}_{\sM,Z}\geq Kg_{\sM}$. Let $N<\infty$. The case $N=\infty$ follows by obvious modifications. 
Consider $\mu_0,\mu_1\in\mathcal{P}^2(M)$ 
that are compactly supported. Otherwise 
we can choose compact exhaustions of $M\times M$ and consider the restriction of optimal couplings to these sets.
There exists a $c$-concave function $\phi$ such that $T_t(x)=\exp_x(-t\nabla \phi_x)$ 
is the unique optimal map between $\mu_0$ and $(T_t)_*\mu_0=\mu_t$, and $\mu_t$ is 
the unique $L^2$-Wasserstein geodesic between $\mu_0$ and $\mu_1$ in $\mathcal{P}^2(M)$. 
$\mu_t$ is compactly supported and $\m_{\sX}$-absolutely continuous \cite{cms}.
\medskip\\
The potential $\phi$ is semi-concave, and by the Bangert-Alexandrov theorem the Hessian $\nabla^2\phi(x)$ exists for $\m_{\sX}$-almost every $x\in M$.
Moreover, for $\m_{\sX}$-almost every $x\in M$ the optimal map $T_t$ admits a Jacobian $DT_t(x)$ for every $t\in[0,1]$. $DT_t(x)$ is non-singular for every $t\in [0,1]$, 
and the Monge-Ampere equation 
\begin{align}\label{jacobi}
\rho_0(x)=\det DT_t(x)\rho_t(T_t(x))
\end{align}holds $\m_{\sX}$-almost everywhere.
\medskip\\
\textbf{2.} We pick a point $x\in M$ where $DT_t(x)$ is non-singular and (\ref{jacobi}) holds. 
For an orthonormal basis $(e_i)_{i=1,\dots,n}$ of $T_xM$
\begin{align*}
t\mapsto V_i(t)=DT_t(x)e_i\in T_{\gamma_x(t)}M
\end{align*}
is the Jacobi field along $\gamma_x(t)=\exp_x(-t\nabla\phi|_x)=T_t(x)$ with initial condition $V_i(0)=e_i$ and $V_i'(0)=-\nabla_{e_i}\nabla \phi|_{x}$.
Therefore
\begin{align*}
V_i''(t)+R(V_i(t),\dot{\gamma}(t))\dot{\gamma}(t)=0 \ \ \mbox{ for } i=1,\dots,n.
\end{align*}
We set $\mathcal{A}_t(x)=(V_1,\dots,V_n)$.
Since $DT_t(x)$ is non-singular for every $t\in[0,1]$, 
Riemannian Jacobi field calculus (for instance see \cite{stugeo2}) yields for $t\mapsto y_t=\log \det \mathcal{A}_t(x)$ the differential inequality
\begin{align}\label{firstone}
y_t''\leq -\frac{1}{n}(y'_t)^2-\ric_{\sX}(\dot{\gamma},\dot{\gamma}).
\end{align}
\textbf{3.} Consider the vector field $Z$, the geodesic $t\in[0,1]\rightarrow \gamma_x(t)$, and the corresponding 
line integral $
t\mapsto \phi^{\sZ}_t(\gamma_x)=:\phi_t.
$
We set $\gamma_x=:\gamma$ and compute
\begin{align*}
\phi''_t=\langle \nabla_{\dot{\gamma}}Z|_{\gamma(t)},\dot{\gamma}(t)\rangle+\langle Z|_{\gamma(t)},\nabla_{\dot{\gamma}}\gamma'(t)\rangle
=\langle \nabla_{\dot{\gamma}}Z|_{\gamma(t)},\dot{\gamma}(t)\rangle=\nabla^sZ(\dot{\gamma},\dot{\gamma}).
\end{align*}
In addition with (\ref{firstone}) and  (\ref{assumption}) this yields
\begin{align*}
y_t''+\phi''_t&\leq -\frac{1}{n}(y'_t)^2-\ric_{\sX}(\dot{\gamma},\dot{\gamma})+\nabla^sZ(\dot{\gamma},\dot{\gamma})\\
&\leq - K|\dot{\gamma}|^2-\frac{1}{N-n}Z\otimes Z(\dot{\gamma},\dot{\gamma})-\frac{1}{n}(y'_t)^2\\
&\leq - K|\dot{\gamma}|^2 -\frac{1}{N}\left(y_t'+Z(\dot{\gamma})(t)\right)^2.
\end{align*}
Hence
\begin{align*}
y_t''+\phi''_t+\frac{1}{N}\left(y_t'+\phi'_t\right)^2+K|\dot{\gamma}|^2\leq 0.
\end{align*}
If we set $\mathcal{I}(t)=e^{y_t+\phi_t}$, then we obtain
\begin{align*}
\frac{d^2}{dt^2}\mathcal{I}^{\frac{1}{N}}\leq - \frac{K|\dot{\gamma}|^2}{N}\mathcal{I}^{\frac{1}{N}},
\end{align*}
or equivalently
\begin{align}\label{secondone}
\mathcal{I}^{\frac{1}{N}}_t\geq \sigma_{\sK,\sN}^{(1-t)}(|\dot{\gamma}|)
\mathcal{I}^{\frac{1}{N}}_0+\sigma_{\sK,\sN}^{(t)}(|\dot{\gamma}|)
\mathcal{I}^{\frac{1}{N}}_1.
\end{align}
Note the dependence on $x\in M$. (\ref{secondone}) holds for $\m_{\sX}$-a.e. $x\in M$.
\medskip\\
\textbf{4.} 
By the same computation as in \cite{stugeo2} we can improve (\ref{secondone}) by taking out the direction of motion if $n\geq 2$.
We know that
\begin{align}\label{ricattii}
U'(t)+U^2(t)+R(t)=0
\end{align}
where $U(t)=\mathcal{A}'(t)\mathcal{A}(t)^{-1}$ and $R_{i,j}(t)=\langle R(V_i,\dot{\gamma})\dot{\gamma},V_j\rangle$. From Lemma 3.1 in \cite{CMS2} one sees that $U$ is symmetric. $R$ has the form
\begin{align*}
R(t)=\begin{pmatrix}
     0 & 0\\
     0 & \bar{R}(t)
     \end{pmatrix}
\end{align*}
for an $(n-1)\times (n-1)$-matrix $\bar{R}$. Hence, if $U=(u_{i,j})_{i,j=1,\dots,n}$, we have
\begin{align}\label{kk}
u_{11}'+\sum_{i=1}^n u_{1,i}^2=0.
\end{align}
Moreover, taking the trace in (\ref{ricattii}) yields
\begin{align}\label{anotherequation}
\tr {U}'+\tr {U}^2 + \ric=0.
\end{align}
where $\ric=\ric(\dot{\gamma})$. 
The Jacobi determinant $\mathcal{J}_t=\det\mathcal{A}_t$ satisfies 
$(\log \mathcal{J})'=\mathcal{J}'/\mathcal{J}=\tr U=u_{11}+\sum_{i=2}^nu_{ii}$. 
Therefore if we set $\lambda_t=\int_0^tu_{11}(s)ds$, we have 
\begin{align*}
y_t=\log\mathcal{J}_t={\lambda_t}+{\int_0^t\Big[\sum_{i=2}^nu_{ii}(s)\Big]ds}
\end{align*}
and (\ref{anotherequation}) becomes $y'' + \tr U^2 +\ric=0$.
$\lambda_t$ describes volume distortion in direction of the transport geodesics. 
We remove this part and consider
$
\bar{y}_t:=y_t-\lambda_t.
$
And if we set $\bar{U}=(u_{ij})_{i,j=2,\dots n}$, then $
\bar{y}_t=\int^t_0\tr \bar{U}ds.
$
A computation yields
\begin{align*}
\bar{y}_t''&=y_t''-\lambda_t''= -\tr U^2(t) -\ric(t)+\sum_{i=1}^n u_{1,i}^2(t)\\ &
\leq
-\tr \bar{U}^2(t)-\ric(t)\leq -\frac{1}{n-1}\left(\tr \bar{U}\right)^2-\ric(t)=-\frac{1}{n-1}(\bar{y}_t')^2-\ric(t).
\end{align*}
and note that $\tr\bar{U}(t)=\bar{y}'_t$. Setting $\bar{\mathcal{I}}_t=e^{\bar{y}_t+\phi_t}$ as in \textbf{3.} we also obtain
\begin{align*}
\frac{d^2}{dt^2}\bar{\mathcal{I}}^{\frac{1}{N-1}}\leq -\frac{K|\dot{\gamma}|^2}{N-1}\bar{\mathcal{I}}^{\frac{1}{N-1}} \ \ \m_{\sX}\mbox{-a.e.}
\end{align*}
if $\ric_{\sM,\sZ}^{\sN}(\dot{\gamma}(t))\geq K$.
\medskip\\
\textbf{5.}
Set $e^{\lambda_t}=L_t$. Note that (\ref{kk}) implies $u_{11}'\leq -u_{11}^2$. Then $L_t''\leq 0$. 
By construction $\mathcal{J}_te^{\phi_t}=\mathcal{I}_t=\bar{\mathcal{I}}_tL_t$. Hence (similar as in part (c) of the proof of Theorem 1.7 in \cite{stugeo2}) we obtain
\begin{align*}
\mathcal{I}_t^{\frac{1}{N}}&=\left(\bar{\mathcal{I}}_tL_t\right)^{\frac{1}{N}}=\Big(\bar{\mathcal{I}}_t^{\frac{1}{N-1}}\Big)^{\frac{N-1}{N}}\left(L_t\right)^{\frac{1}{N}}
\geq 
\tau_{\sK,\sN}^{(1-t)}(|\dot{\gamma}|)\mathcal{I}_0^{\frac{1}{N}}+\tau_{\sK,\sN}^{(t)}(|\dot{\gamma}|)\mathcal{I}_1^{\frac{1}{N}}.
\end{align*}
or 
\begin{align*}
\mathcal{J}_t^{\frac{1}{N}}e^{\frac{1}{N}\phi_t}\geq \tau_{\sK,\sN}^{(1-t)}(|\dot{\gamma}|)\mathcal{J}_0^{\frac{1}{N}}+\tau_{\sK,\sN}^{(t)}(|\dot{\gamma}|)e^{\frac{1}{N}\phi_1}\mathcal{J}_1^{\frac{1}{N}}\ \ \m_{\sX}\mbox{-a.e.}
\end{align*}
Consider $\mu_0$, $\mu_1$, $T_t$ and $\mu_t=\rho_td\vol_g$ as in \textbf{1.}. Then for $\m_{\sX}$-a.e. $x\in M$ we have
\begin{align*}
\left[\rho_t(T_t(x))e^{-\phi_t(\gamma_x)}\right]^{-\frac{1}{N}}
&=\rho_0(x)^{-\frac{1}{N}}\left[\mathcal{J}_t(x)^{}e^{\phi_t(\gamma_x)}\right]^{\frac{1}{N}}\\
&\geq \tau_{\sK,\sN}^{(1-t)}(|\dot{\gamma}_x|)\rho_0(x)^{-\frac{1}{N}}\mathcal{J}_0(x)^{\frac{1}{N}}\\
&\hspace{2cm}+\tau_{\sK,\sN}^{(t)}(|\dot{\gamma}_x|)\left[e^{\phi_1(\gamma_x)}\rho_0(x)^{-1}\mathcal{J}_1(x)\right]^{\frac{1}{N}}\\
&=  \tau_{\sK,\sN}^{(1-t)}(|\dot{\gamma}_x|)\rho_0(x)^{-\frac{1}{N}}\\
&\hspace{2cm}+\tau_{\sK,\sN}^{(t)}(|\dot{\gamma}_x|)\left[e^{-\phi_1}\rho_1(T_1(x))\right]^{-\frac{1}{N}}.
\end{align*}
Recall that $\mu_t=(T_t)_{\star}\mu_0$.
$\Pi$ is the unique dynamical optimal coupling between $\mu_0$ and $\mu_1$.
Hence, integration 
with respect to $\mu_0$ yields
\begin{align*}&
S_{\sN,t}^Z(\Pi)=
-\int\left[\rho_t(\gamma_t)e^{-\phi_t(\gamma)}\right]^{-\frac{1}{N}}d\Pi(\gamma)=-\int\left[\rho_t(T_t(x))e^{-\phi_t(\gamma_x)}\right]^{-\frac{1}{N}}d\mu_0(x)\\
&\leq -\int \Big[\tau_{\sK,\sN}^{(1-t)}(|\dot{\gamma}_x|)\rho_0(x)^{-\frac{1}{N}}+\tau_{\sK,\sN}^{(t)}(|\dot{\gamma}_x|)\left[e^{-\phi_1(\gamma_x)}\rho_1(T_1(x))\right]^{-\frac{1}{N}}\Big]d\mu_0(x)\\
&= -\int \Big[\tau_{\sK,\sN}^{(1-t)}(|\dot{\gamma}|)\rho_0(\gamma_0)^{-\frac{1}{N}}+\tau_{\sK,\sN}^{(t)}(|\dot{\gamma}|)e^{-\frac{1}{N}\phi_1(\gamma)}\rho_1(\gamma_1)^{-\frac{1}{N}}\Big]d\Pi(\gamma).
\end{align*}
That is the claim.
\medskip\\
\textbf{6.} ``$\Longleftarrow$'': We argue by contradiction. Assume $(X,\de_{\sX},\m_{\sX},Z)$ satisfies the curvature-dimension condition $CD(K,N)$. We only consider the case $N>n$. 
Since $\sigma_{\sK,\sN}^{\scriptscriptstyle{(t)}}(\theta)\leq \tau_{\sK,\sN}^{\scriptscriptstyle{(t)}}(\theta)$, $(X,\de_{\sX},\m_{\sX},Z)$ satisfies the 
condition $CD^*(K,N)$ as well. Assume there is $v\in T_xM$ such that $\ric_{\sM,\sZ}^{\sN}(v)\leq (K-6\epsilon)|v|^2$. Choose a smooth function $\psi$ such that
\begin{align*}
\nabla \psi(x)=v \ \ \ \& \ \ \ \nabla^2\psi(x)=\frac{1}{N-n}\langle Z,v\rangle|_x I_n.
\end{align*}
We set $\lambda:=\frac{1}{N-n}\langle Z,v\rangle|_x$. It follows $\Delta \psi(x)= - n\lambda$.
We can assume that $\psi$ has compact support in $B_1(x)$, and $\psi$ and $-\psi$ are Kantorovich potentials by replacing $\psi$ by $\theta\psi$ and $v$ by $\theta v$ for a sufficiently small $\theta >0$ 
(we apply Theorem 13.5 in \cite{viltot} to $\psi$ and $-\psi$).
Then the map $T_t(y)=\exp_y(-t\nabla\psi(y))$ for $t\in [-\delta,\delta]$ induces a constant speed $L^2$-Wasserstein geodesic $(\mu_t)_{t\in [-\delta,\delta]}$. 
For this particular $T_t$ we repeat calculations of the previous steps and obtain for $\mathcal{A}'_t\mathcal{A}_t^{-1}={U}_t$
\begin{align*}
\mbox{tr}{U}_t'+\phi_t''=-\mbox{tr}{U}_t^2-\ric(\dot{\gamma}_t,\dot{\gamma}_t)+\nabla^sZ(\dot{\gamma}_t,\dot{\gamma}_t)
\end{align*}
$\m_{\sX}$-a.e.
where $\gamma_t=T_{t}(x)$. Now, we consider $\mu=\m_{\sX}(B_{\eta}(x))^{-1}\m_{\sX}|_{B_{\eta}(x)}$ and the induced Wasserstein geodesic $(T_t)_{\star}\mu=\mu_t$. We set $T_t(y)=\gamma_y(t)$, and note
that $\dot{\gamma}_x(0)=v$.
By continuity of $\ric_{\sM,\sZ}^{\sN}(\cdot)$ we can choose $\eta$ and $\delta>0$ such that 
\begin{align}\label{lolo}
\ric_{\sM,\sZ}^{\sN}(\dot{\gamma}_y(t))\leq (K-5\epsilon)|\dot{\gamma}_y(t)|^2 \mbox{ for }y\in B_{\eta}(x),\ t\in [-\delta,\delta].
\end{align}
Moreover, by continuity we can choose $\eta$ and $\delta>0$ even smaller such that $\frac{1}{2}|v|\leq |\dot{\gamma}_y(t)|$ for $y\in B_{\eta}(x)$ and $t\in [-\delta,\delta]$.
And again by using just continuity of $Z$, $\psi$ and all its first and second order derivatives for $\epsilon'\leq \frac{1}{4}\epsilon|v|^2$ we can choose $\eta>0$ and $\delta>0$ even more smaller such that 
\begin{align*}
&(a)\ |\mbox{tr}U^2_t-n\lambda^2|<\epsilon', \ \ (b)\ \left|\frac{1}{N}(\mbox{tr}U_t+\langle Z,\dot{\gamma}_t\rangle)^2-\frac{1}{N}(n\lambda+\langle Z,\dot{\gamma}_t\rangle)^2\right|<\epsilon',\\
&\ \ \ \ \ \ \ \ \ \ \ \ (c)\ {\textstyle \frac{n}{N(N-n)}}\left|(N-n)\lambda -\langle Z,\dot{\gamma}_t\rangle\right|^2<\epsilon'
\end{align*}
for every $t \in (-\delta,\delta)$ and every $y\in B_{\eta}(x)$. Note that $\mbox{tr}(\nabla^2\psi(x))^2=n\lambda^2$ and $\Delta\psi(x)=n\lambda$.
Then we compute for $y_t=\log\det\mathcal{A}_t$ - omitting the dependency on $y\in B_{\eta}(x)$, and using $\epsilon'\leq \epsilon\frac{1}{4}|v|^2\leq \epsilon|\dot{\gamma}_y(t)|^2$ for any $y\in B_{\eta}(x)$ and each $t\in [-\delta,\delta]$ - 
\begin{align*}
y_t''+\phi_t''=\mbox{tr}{U}_t'+\phi_t''&= -\mbox{tr}U_t^2-\ric_{\sM}(\dot{\gamma}(t),\dot{\gamma}(t))+\nabla^sZ(\dot{\gamma}(t),\dot{\gamma}(t))\\
&\overset{(\ref{lolo})+(a)}{\geq} -\frac{1}{n}\left(n\lambda\right)^2 -(K-4\epsilon)|\dot{\gamma}_t|^2- \frac{1}{N-n}\langle Z,\dot{\gamma}_t\rangle^2
\\
&=-\frac{1}{N}\left({n}\lambda+\langle Z,\dot{\gamma}_t\rangle\right)^2 -(K-4\epsilon)|\dot{\gamma}_t|^2\\
&\hspace{2cm}- \frac{n}{N(N-n)}\left(\langle Z,\dot{\gamma}_t\rangle -{(N-n)}\lambda\right)^2\\
&\overset{(c)}{\geq} -\frac{1}{N}\left({n}\lambda+\langle Z,\dot{\gamma}_t\rangle\right)^2- (K-3\epsilon)\left|\dot{\gamma}_t\right|^2\\
&\overset{(b)}{\geq} -\frac{1}{N}\left(\mbox{tr}U_t+ \phi_t'\right)^2- (K-2\epsilon)\left|\dot{\gamma}_t\right|^2\\
&= -\frac{1}{N}\left(y_t'+ \phi_t'\right)^2- (K-2\epsilon)\left|\dot{\gamma}_t\right|^2\ \ \mbox{ for }t\in[-\delta,\delta]
\end{align*}
where we use $$\frac{1}{N}(a+b)^2+\frac{n}{N(N-n)}\left(b-a\frac{N-n}{n}\right)^2=\frac{1}{n}a^2+\frac{1}{N-n}b^2$$
in the third equality.
Recall $\mbox{tr}U_t=y_t'$. The previous differential inequality is equivalent to 
\begin{align*}
\left[e^{\frac{1}{N}(y_t+\phi_t)}\right]''\geq - \frac{K-2\epsilon}{N}\left|\dot{\gamma}_t\right|^2e^{\frac{1}{N}(y_t+\phi_t)} \ \ \mbox{ on }[-\delta,\delta].
\end{align*}
A reparametrization $s\in[-\frac{1}{2},\frac{1}{2}]\rightarrow -s\delta+s\delta$ yields
\begin{align}\label{lala}
\left[e^{\frac{1}{N}(\hat{y}_t+\hat{\phi}_t)}\right]''\geq - \frac{K-2\epsilon}{N}\left|\dot{\hat{\gamma}}_t\right|^2e^{\frac{1}{N}(\hat{y}_t+\hat{\phi}_t)} \ \ \mbox{ on }[-\textstyle{\frac{1}{2}},\textstyle{\frac{1}{2}}].
\end{align}
where $\hat{y}_t+\hat{\phi}_t={y}_{-\delta t+\delta t}+{\phi}_{-\delta t+\delta t}$, $\hat{\gamma}_t=\exp(-t\delta\nabla\phi)$ and $|\dot{\hat{\gamma}}_t|=2\delta|\dot{\gamma}_t|=L(\hat{\gamma})$.
$\delta\phi$ is a Kantorovich potential for the Wasserstein geodesic $t\in [-\frac{1}{2},\frac{1}{2}]\mapsto \mu_{-\delta t+\delta t}$, and by the previous computation (\ref{lala}) is 
the differential inequality as in step (4) for the potential $\delta \phi$ but with reversed inequalities. 
We obtain (\ref{secondone}) with reverse inequalities and $K$ replaced by $K-2\epsilon$ that contradicts the condition $CD^*(K,N)$.
\end{proof}
\section{Geometric consequences}
\noindent
Consider a generalized smooth metric measure space $(X,\de_{\sX},\m_{\sX})$ with regular set $M$.
In this section we assume that $\vol_{\sM}$ is $\m_{\sX}$-absolutely continuous.
We consider again vector fields instead of $1$-forms.
\noindent
\smallskip\\
Recall the definition of the measure $\mathcal{M}$ from section 2. For any measurable set $A\subset X$ with $\m_{\sX}(A)>0$ and any point $x_0\in X$ we define $\mathcal{M}_{x_0,A}=\Psi_{\star}(\delta_{x_0}\otimes \m_{\sX}|_A)$
where $\Psi(x,y)=\gamma_{x,y}:[0,1]\rightarrow X\in\mathcal{G}(X)$. Since geodesics are $\m_{\sX}\otimes\m_{\sX}$-almost everywhere unique, $\Pi_{x_0,A}=\m_{\sX}(A)^{-1}\mathcal{M}_{x_0,A}$ is the unique optimal dynamical plan between 
$\delta_{x_0}$ and $\m_{\sX}(A)^{-1}\m_{\sX}|_A$.
\smallskip\\
We define for $r\in (0,\infty)$ and $\bar{B}_r(x_0)=\overline{B_r(x_0)}$
$$
v^{\sZ,x_0}(r):=\int e^{\phi(\gamma)}\mathcal{M}_{x_0,\bar{B}_r(x_0)}(d\gamma)
=\int e^{\phi(\gamma_{x,y})}d(\delta_{x_0}\otimes\m_{\sX}|_{\bar{B}_r(x_0)})(x,y)
$$ 
If $Z=0$, then $v^{\sZ,x_0}(r)=\m_{\sX}(\bar{B}_r(x_0))=:v^{x_0}(r)$.
\medskip\\
We say the vector field $Z$ has \textit{locally finite flow} if for any $x_0\in X$ there is $r>0$ such that $v^{\sZ,x_0}(r)<\infty$.
%
If $Z=\nabla f$ for some smooth $f$, $Z$ has locally finite flow if $e^{-f}\m_{\sX}$ is a locally finite measure.
\medskip\\
We define for $r\in (0,\infty)$
$$
s^{\sZ,x_0}(r):=
\limsup_{\delta\rightarrow 0}{\delta}^{-1}\int e^{\phi(\gamma_{x,y})}\delta_x\otimes\m_{\sX}|_{\bar{B}_{r+\delta}(x_0)\backslash B_r(x_0)}.
$$
If $Z=0$, then $s^{\sZ,x_0}(r)=:s^{x_0}(r)$.
\begin{theorem}[Generalized Bishop-Gromov inequality]\label{bg} Let $(X,\de_{\sX},\m_{\sX})=\X$ be a generalized smooth metric measure space, and $Z$ an $L^2$-integrable vector field. Assume
 $(\X,Z)$ satisfies the curvature-dimension $CD(K,N)$ for $K\in \mathbb{R}$ and $N\geq 1$. 
Then each bounded set $X'\subset \supp\m_{\sX}$ has finite $\m_{\sX}$-measure, and either $\m_{\sX}$ is supported by one point or all points and all spheres have mass $0$.
\medskip\\
Moreover, if $N>1$ then for each fixed $x_0\in \supp[\m_{\sX}]\cap M$ and all $0<r<R$ (with $R\leq \pi \sqrt{\textstyle{\frac{N-1}{K}}}$ if $K>0$), we have
\begin{align}\label{spheres}
\frac{s^{\sZ,x_0}(r)}{s^{\sZ,x_0}(R)}\geq \frac{\sin_{\sK/(\sN-1)}^{\sN-1}(r)}{\sin_{\sK/(\sN-1)}^{\sN-1}(R)}
\end{align}
and 
\begin{align}\label{balls}
\frac{v^{\sZ,x_0}(r)}{v^{\sZ,x_0}(R)}\geq \frac{\int_0^r\sin_{\sK/(\sN-1)}^{\sN-1}(t)dt}{\int^R_0\sin_{\sK/(\sN-1)}^{\sN-1}(t)dt}
\end{align}
In particular, if $K=0$, then
\begin{align*}
\frac{s^{\sZ,x_0}(r)}{s^{\sZ,x_0}(R)}\geq \frac{r^{\sN-1}}{R^{\sN-1}}\
\
\&\
\
\frac{v^{\sZ,x_0}(r)}{v^{\sZ,x_0}(R)}\geq \frac{r^N}{R^N}
\end{align*}
for all $R,r>0$ and $x_0$, and the latter also holds for $N=1$ and $K\leq 0$.
\end{theorem}
\begin{proof}
Fix a point $x_0\in \supp[\m_{\sX}]\cap M$, and assume first that $\m_{\sX}(\left\{x_0\right\})=0$. We
set $A_{R}={\bar{B}_{R+\delta R}(x_0)}\backslash B_R(x_0)$, and $t\in[0,1]\mapsto (e_t)_{\star}\Pi_{x_0,A_R}=\mu_t\in \mathcal{P}^2(X)$.
$(\mu_t)_{t\in[0,1]}$ is the unique geodesic between between its endpoints. Note that by definition of a generalized smooth metric measure space transport geodesics 
connecting $x_0\in M$ and points in $A_R$ are $\Pi_{x_0,A}$-almost surely contained in $M$. 
Therefore, since $x_0\in M$ by Riemannian calculus $\mu_t\in \mathcal{P}^2(\vol_{\sM})$ for any $t\in (0,1)$, and
since $\vol_{\sM}$ is $\m_{\sX}$-absolutely continuous, it follows that $\mu_t\in\mathcal{P}^2(\m_{\sX})$ for any $t\in (0,1)$. In particular,
$s\mapsto (e_{(1-s)t_0+st_1})_{\star}\Pi_{x_0,A_R}=:\mu_s$ is the unique Wasserstein geodesic between $(e_{t_0})_{\star}\Pi_{x_0,R}$ and $(e_{t_1})_{\star}\Pi_{x_0,R}$ for any $t_0,t_1\in (0,1]$. 
The unique optimal dynamical plan $\tilde{\Pi}$ is obtained by pushforward $\Pi_{x_0,A_R}$ w.r.t. the map $\gamma\mapsto \tilde{\gamma}$ with $\tilde{\gamma}(s)=\gamma((1-s)t_0+st_1)$. 
Similar $\tilde{\rho}_s=\rho_{(1-s)t_0+st_1}$. 
\smallskip\\
First, we check continuity of $t\mapsto\left[\rho_t(\gamma_t)e^{-\phi_{{s}}(\gamma)}\right]^{-\frac{1}{N}}=:h_t(\gamma)$ on $(0,1)$.
By the condition $CD^*(K,N)$ and Theorem \ref{nonbranching} we have
\begin{align*}
\left[\tilde{\rho}_{{s}}(\gamma_{{s}})e^{-\phi_{{s}}(\gamma)}\right]^{-\frac{1}{N}}
&\geq \sigma_{\sK,\sN}^{(1-s)}(|\dot{\gamma}|)\tilde{\rho}_{0}(\gamma_{0})^{-\frac{1}{N}}+\sigma_{\sK,\sN}^{(s)}(|\dot{\gamma}|)\left[e^{-\phi_1(\gamma)}\tilde{\rho}_1(\gamma_1)\right]^{-\frac{1}{N}}
\end{align*}
for $s\in[0,1]$ and $\tilde{\Pi}$-almost every $\gamma\in\mathcal{G}(X)$. That is
\begin{align*}
\frac{d^2}{dt^2}\left[\rho_t(\gamma_t)e^{-\phi_{{s}}(\gamma)}\right]^{-\frac{1}{N}}\leq -\frac{K}{N}|\dot{\gamma}|^2\left[\rho_t(\gamma_t)e^{-\phi_{{s}}(\gamma)}\right]^{-\frac{1}{N}}\ \ \mbox{on }(0,1)
\end{align*}
for ${\Pi}$-almost every $\gamma\in\mathcal{G}(X)$. In particular, $h_t(\gamma)$ is semi-convex and therefore continuous on $(0,1)$ for $\Pi$-almost every $\gamma$.
\smallskip\\
Now, we choose $t_0=\epsilon$ and $t_1=1$, and by the condition $CD(K,N)$ and again Theorem \ref{nonbranching} we have
\begin{align*}
\left[\tilde{\rho}_{{s}}(\gamma_{{s}})e^{-\phi_{{s}}(\gamma)}\right]^{-\frac{1}{N}}
&\geq \tau_{\sK,\sN}^{(1-{s})}(|\dot{\gamma}|)\tilde{\rho}_{0}(\gamma_{0})^{-\frac{1}{N}}+\tau_{\sK,\sN}^{(s)}(|\dot{\gamma}|)\left[e^{-\phi_1(\gamma)}\tilde{\rho}_1(\gamma_1)\right]^{-\frac{1}{N}}\\
&\geq \tau_{\sK,\sN}^{(s)}(|\dot{\gamma}|)\left[e^{-\phi_1(\gamma)}\rho_1(\gamma_1)\right]^{-\frac{1}{N}}\ \ \mbox{for }\tilde{\Pi}\mbox{-a.e. }\gamma\in\mathcal{G}(X),
\end{align*}
or equivalently,
for ${\Pi}$-almost every $\gamma\in\mathcal{G}(X)$:
\begin{align*}
\left[{\rho}_{{(1-s)\epsilon + s}}(\gamma_{(1-s)\epsilon+s})e^{-\phi_{(1-s)\epsilon+s}(\gamma)}\right]^{-\frac{1}{N}}\geq \tau_{\sK,\sN}^{(s)}((1-\epsilon)|\dot{\gamma}|)\left[e^{-\phi_1(\gamma)}\rho_1(\gamma_1)\right]^{-\frac{1}{N}}.
\end{align*}
By continuity of $h_t(\gamma)$, $\phi_{s}(\gamma)$ in $(0,1)$, it follows
for $\epsilon\rightarrow 0$
\begin{align*}
 \rho_{{s}}(\gamma_{{s}})^{-1}e^{\phi_{{s}}(\gamma)}
&\geq \tau_{\sK,\sN}^{(s)}(|\dot{\gamma}|)^{N}e^{\phi_1(\gamma)}\rho_1(\gamma_1)^{-1}.
\end{align*}
Now, we choose $s=r/R$ for $r\in (0,R)$. Integration with respect to $\Pi$ yields
\begin{align}\label{lele}
\int \rho_{{r/R}}(\gamma_{{r/R}})^{-1}e^{\phi_{{r/R}}(\gamma)}d\Pi(\gamma)
&\geq  \tau_{\sK,\sN}^{(r/R)}(R\pm R\delta)^{\sN}\int e^{\phi_1(\gamma)}d\mathcal{M}_{x_0,A_{R}(x_0)}(\gamma)
\end{align}
where we choose $-$ in $\tau$ if $K\geq 0$, and $+$ otherwise.\\
\\
Note that the left hand side of the previous inequality can be rewritten as follows
\begin{align*}
\int \rho_{{r/R}}(\gamma_{{1}})^{-1}e^{\phi_{{1}}(\gamma)}d\bar{\Pi}(\gamma)&=
\int \rho_{{r/R}}(\gamma_{{1}})^{-1}e^{\phi_{{1}}(\gamma_{x_0,\gamma_1})}d\bar{\Pi}(\gamma)\\
&=\int \rho_{r/R}^{-1}(x)e^{\phi_1(\gamma_{x_0,y})}d(e_{r/R})_{\star}\Pi(y)\\
&=\int e^{\phi_1(\Psi(x_0,y))}d\m_{\sX}|_{\supp\mu_{r/R}}(y)\\
&= \int e^{\phi_1(\Psi(x,y))}d\delta_{x_0}\otimes\m_{\sX}|_{\supp\mu_{r/R}}(x,y)
\end{align*}
where $\bar{\Pi}$ is the optimal plan between $\delta_{x_0}$ and $(e_{r/R})_{\star}\Pi$. $\bar{\Pi}$ arises as the pushforward w.r.t. the map $\Phi(\gamma)(t)=\gamma(tr/R)\in\mathcal{G}(X)$.
Since $\gamma_{r/R}\in A_{r}(x_0)$ for $\Pi$-almost every $\gamma$, the left hand side in (\ref{lele}) is dominated by $\int e^{\phi_{1}(\gamma)}d\mathcal{M}_{x_0,A_{r}(x_0)}$, and we obtain the following 
inequality:
\begin{align}\label{differences}
\frac{v^{x_0,Z}(r+\delta r)-v^{x_0,Z}(r)}{\delta r} \geq \frac{\sin^{\sN-1}_{\sK/(\sN-1)}(r\pm r\delta)}{\sin^{\sN-1}_{\sK/(\sN-1)}(R\pm R\delta)}
\frac{v^{x_0,Z}(R+\delta R)-v^{x_0,Z}(R)}{\delta R}.
\end{align}
By construction, $r\mapsto v^{\sZ,x_0}(r)$ is monoton non-decreasing and right continuous. Hence, it has only countably many discontinuities, 
and we can pick an arbitrarily small continuity point $r>0$. Then, for arbitrarily small $\delta>0$ we can apply (\ref{differences}) to deduce that $v^{\sZ,x_0}$ is continuous at any $R>r$.
Hence $v^{\sZ,x_0}$ is continuous on $(0,\infty)$.
In particular, for any $r>$ and $x_0\in X$ it follow that
$$\int e^{\phi_t(\gamma)}d\mathcal{M}_{x_0,\partial B_R(x_0)}(\gamma)=\int_{\partial B_r(x_0)} e^{\phi_t(\gamma_{x_0,y})}d\m_{\sX}(y)=0,$$
and $\int_{\left\{y_0\right\}} e^{\phi_t(\gamma_{x_0,y})}d\m_{\sX}(y)=0$ for any $y\in X$. Hence, $\m_{\sX}(\partial B_r(x_0))=\m_{\sX}(\left\{x_0\right\})=0$ for any $r>0$ and $x_0\in X$.
Moreover, if $\delta\downarrow 0$ in (\ref{differences}), we obtain (\ref{spheres}).
\medskip\\
Now, for $r>0$ and $\delta>0$ given, one considers $2^n$ points $(1+2^{-n}\delta)r$ in $[r,(1+\delta)r)$. By (\ref{differences}) again it follows
\begin{align*}
0\leq v^{\sZ,x_0}((1+2^{-n}\delta)r)-v^{\sZ,x_0}(r) \leq C(\delta,r)2^{-n}\delta r.
\end{align*}
This implies $v^{\sZ,x_0}$ is locally Lipschitz, and therefore weakly differentiable, and its weak derivative is given by $s^{\sZ,x_0}(r)$ $\mathcal{L}^1$-almost everywhere. Then we can apply Gromov's integration trick 
\cite[Lemma 3.1]{chavelmodern} and obtain (\ref{balls}).
\medskip\\
It remains to consider the case $\m_{\sX}(\left\{x_0\right\})>0$. Assume $\supp\m_{\sX}\neq \left\{x_0\right\}$. Since $\supp\m_{\sX}$ is a length space, 
there is a least one point $y_0\in M$ such that $\m_{\sX}(y_0)=0$. We then apply apply the previous steps. 
This yields $\m_{\sX}(x)=0$ for every $x\in X\backslash \left\{y_0\right\}$ which is a contradiction.
Otherwise $X=\left\{x_0\right\}$ and $\m_{\sX}\equiv \delta_{x_0}$.
\end{proof}
\begin{theorem}[Generalized Bonnet-Myers]
Assume that $(X,\de_{\sX},\m_{\sX},Z)$  satisfies the curvature-dimension $CD(K,N)$ for $K>0$, $N\geq 1$ and a $L^2$-integrable vectorfield $Z$. Then
the support of $\m_{\sX}$ is compact and its diameter $L$ satifies
\begin{align*}
L\leq \pi\sqrt{\frac{N-1}{K}}.
\end{align*}
In particular, if $N=1$, then $\supp\m_{\sX}$ consist of one point. 
\end{theorem}
\begin{proof}We argue by contradiction. 
Assume there are points $x_0,x_1\in \supp\m_{\sX}$ with $\de_{\sX}(x_0,x_1)>\pi\sqrt{\scriptstyle{\frac{N-1}{K}}}.$ Since $\m_{\sX}(X\backslash M)=0$, we can assume that $x_0\in M$.
Choose $\epsilon >0$ such that $\theta:=\de_{\sX}(\delta_{x_0},\bar{B}_{\epsilon}(x_1))>\pi\sqrt{\scriptstyle{\frac{N-1}{K}}}$, and 
set $\mu={\scriptstyle \m_{\sX}(\bar{B}_{\epsilon}(x_1))^{-1}}\m_{\sX}|_{\bar{B}_{\epsilon}(x_1)}$.
Let $\Pi_{x_0,\mu}$ be the optimal transport between $\delta_{x_0}$ and $\mu$. As in the proof of the Bishop-Gromov comparison $x_0\in M$ implies that $(e_t)_{\star}\Pi_{x_0,\mu}$ is $\m_{\sX}$-absolutely continuous.
Hence
\begin{align*}
\int \rho_{{1/2}}(\gamma_{\sscr{\frac{1}{2}}})^{-1}e^{\phi_{\sscr{\frac{1}{2}}}(\gamma)}d\Pi(\gamma)\geq  \tau_{\sK,\sN}^{\sscr{(\frac{1}{2})}}(\theta)^{\sN}\int e^{\phi_1(\gamma)}d\mathcal{M}_{x_0,\bar{B}_{\epsilon}(x_1)}(\gamma)=\infty.
\end{align*}
The left hand side is dominated by 
$\int e^{\phi_{{1}}(\gamma)}d\mathcal{M}_{x_0,\bar{B}_{2\theta}(x_0)}$.
Hence, $v^{\sZ,x_0}(2\theta)=\infty$. On the other hand, Theorem \ref{bg} implies that $v^{\sZ,x_0}(r)<\infty$ for every $r>0$.
\end{proof}
\noindent
For $(X,\de_{\sX},\m_{\sX})$ and $Z$ we define
\begin{align*}
\sup_{B_{\epsilon}(x_0)\in\mathcal{B}_{\epsilon},\epsilon>0} \sum_{B_{\epsilon}(x_0)\in \mathcal{B}_{\epsilon}}^{}\int e^{\phi_{{1}}(\gamma)}d\mathcal{M}_{x_0,\bar{B}_{\epsilon}(x_0)}(\gamma)=:\M_{X,Z}\in [0,\infty]
\end{align*}
where the supremum is w.r.t. families $\mathcal{B}_{\epsilon}$ of disjoint $\epsilon$-Balls.
We also define
\begin{align*}
 \inf_{x_0\in X}\int e^{\phi_{{1}}(\gamma)}d\mathcal{M}_{x_0,X}(\gamma)=:\m_{X,Z}\in [0,\infty].
\end{align*}
We have $\M_{X,Z}\geq \m_{X,Z}$, and 
if $Z$ is essentially bounded, and if $\diam_{\sX}=D<\infty$, then
$\M_{\sX,\sZ}\leq e^{\left\|Z\right\|_{L^{\infty}}D}\m_{\sX}(X)<\infty$ and $\m_{X,Z}> e^{-\left\|Z\right\|_{L^{\infty}}D}\m_{\sX}(X)>0$.
\\
\\
\noindent
Let $\mathcal{X}(K,N)$ be the family of smooth metric spaces $(X,\de_{\sX})$ such that there exits a measure $\m_{\sX}$ and an $L^2$-integrable vector field $Z$ such that $(X,\de_{\sX},\m_{\sX},Z)$ satisfies the condition $CD(K,N)$.  
Let $\mathcal{X}(K,N,L)$ be the subset of metric spaces $(X,\de_{\sX})$ in $\mathcal{X}(K,N)$ with $\diam_{\sX}\leq L$. Note that $\mathcal{X}(K,N)$ usually denotes the class of \textit{metric measure spaces} that satisfy a condition 
$CD(K,N)$. 
\begin{corollary}
The family $\mathcal{X}(K,N,L,C)$ of smooth metric spaces $(X,\de_{\sX})$ such that $(X,\de_{\sX},\m_{\sX},Z)$ satisfies $CD(K,N)$ for a measure $\m_{\sX}$ and an $L^2$-integrable vector field $Z$ with
\begin{align*}
 {\frac{\M_{X,Z}}{\m_{X,Z}}}\leq C<\infty
\end{align*}
is precompact with respect to Gromov-Hausdorff convergence.
\end{corollary}
\begin{proof}Let $(X,\de_{\sX})\in\mathcal{X}(K,N,L,C)$, and let $\m_{\sX}$ be a measure and let $Z$ be an admissible vector field such that ${\scriptstyle\frac{\M_{X,Z}}{\m_{X,Z}}}\leq C <\infty$
and $(X,\de_{\sX},\m_{\sX},Z)$ satifies $CD(K,N)$.
Fix $\epsilon>0$ and a family $\mathcal{B}_{\epsilon}$ of disjoint $\epsilon$-balls in $X$. We replace every $\epsilon$-ball by a slightly smaller $\epsilon'$-ball that is contained in the first one and centered in $M$ for some $\epsilon'<\epsilon$.
This is possible since $M$ is dense in $X$.
The generalized Bishop-Gromov inequality yields for every $B_{\epsilon'}(x_0)$ with $x_0\in \supp\m_{\sX}\cap M$.
\begin{align*}
\int e^{\phi_{{1}}(\gamma)}d\mathcal{M}_{x_0,\bar{B}_{\epsilon'}(x_0)}(\gamma)\geq C_{K,N,L} \int e^{\phi_{{1}}(\gamma)}d\mathcal{M}_{x_0,X}(\gamma)\geq C_{K,N,L}\m_{X,Z}.
\end{align*}
Hence $\infty>C\geq \frac{\M_{X,Z}}{\m_{X,Z}}\geq C(K,N,L)\#\mathcal{B}_{\epsilon}$, and
the family $\mathcal{X}(K,N,L,C)$ is uniformily totally bounded, and the claim follows from Gromov's (pre)compactness theorem.
\end{proof}
\begin{remark}
Let us emphasize that the measures $\m_{\sX}$ do not play a role explicitly in proving that the family $\mathcal{X}(K,N,L,C)$ is uniformily totally bounded. The total masses do not need to be uniformly bounded, but the quantity $
 {{\M_{X,Z}}/{\m_{X,Z}}}$. Therefore, a sequence of measures has not to be precompact, and consequently the statement is not formulated in terms of measured Gromov-Hausdorff convergence (or measured Gromov convergence \cite{gmsstability}).
 This viewpoint becomes even more apparent if one considers just vector fields $Z$ that are uniformily essentially bounded. Then, in ${{\M_{X,Z}}/{\m_{X,Z}}}$ the mass cancels, and $\left\|Z\right\|_{L^{\infty}}\leq C$ suggests that one 
 can extract a ``weakly converging'' subseqence of $\left\{Z\right\}$.
\end{remark}

\section{Examples}
\paragraph{\textbf{Riemannian manifolds with boundary}} Let $(M,g_{\sM})$ be a geodesically convex, compact Riemannian manifold with non-empty boundary, and $\dim_{\sM}=n$.
For instance, one can consider a convex domain in $\mathbb{R}^n$. Condition (\ref{transportcondition}) is obviously satisfied.
Note that any Wasserstein geodesic in $\mathcal{P}^2(M)$ between absolutely continuous probability measures with bounded densities has midpoint measures with bounded densities since $\ric_{\sM}\geq -C$ for some $C>0$.
Let $Z$ be any smooth vector field on $M$ that is $L^2$-integrable w.r.t. $\vol_{\sM}$. Assume that $-\nabla^s Z\geq K+ \frac{1}{N}Z\otimes Z$ for some $K\in\mathbb{R}$ and $N\geq 1$.
Then $(M,\de_{\sM},\vol_{\sM},Z)$ satisfies the condition $CD(K+K',N+n)$ if $\ric_{\sM}\geq K'$. 
In particular, if $K=0$, $\ric_{\sM}\geq K'+\epsilon>0$ and $\left\|Z\right\|_{L^{\infty}}\leq \sqrt{\epsilon N}$, then $\nabla^sZ\geq K-\epsilon$ and 
$(M,\de_{\sM},\vol_{\sM},Z)$ satisfies the condition $CD(K',N+n)$.
\medskip
\noindent
\paragraph{\textbf{$N$-warped products}}
Let $(B,g_B)$ be a $d$-dimensional Riemannian manifold, and lef $f:B\rightarrow (0,\infty)$ be a smooth function.
We assume
\begin{itemize}\smallskip
 \item[(i)] $\ric_{\sB}\geq (d-1)Kg_{\sB}$\smallskip
 \item[(ii)] $\nabla^2f+Kfg_B\leq 0$\smallskip
 \item[(iii)] $|\nabla f|^2+Kf^2\leq K_{\scriptscriptstyle{F}}$.
\end{itemize}
\medskip
Let $(F,g_{\sF},\m_{\scriptscriptstyle{F}})=\F$ be weighted Riemannian manifold where $\m_{\scriptscriptstyle{F}}=e^{-\Psi}d\vol_{\scriptscriptstyle{F}}$ is a smooth 
reference measure and $\vol_{\scriptscriptstyle{F}}$ is the Riemannian volume with respect to $g_{\scriptscriptstyle{F}}$. The
Riemannian $N$-warped product $C:=B\times^N_fF$ between $B$, $\F$ and $f$ is given by the metric completion $(X,\de_{\sX})$ of the Riemannian metric
\begin{eqnarray*}
g_{\scriptscriptstyle{B\times_f F}}=(p_{\sB})^{\star}g_{\sB}+(f\circ p_{\sF})^2(p_{\scriptscriptstyle{F}})^{\star}g_{\scriptscriptstyle{F}}.
\end{eqnarray*}
on $B\times F$, and the measure $d\vol_{\sB}\otimes f^Nd\m_{\sF}=\m^N_f$ on $X$ (see also \cite{ketterer}). $p_B:B\times F\rightarrow B$ and $p_F:B\times F\rightarrow F$ are the projections maps.
\medskip\\
The $N$-Ricci tensor of $g_{\scriptscriptstyle{B\times_f F}}$ and $\m_{f}^{\sN}$ is computed in \cite{ketterer}:
\begin{align*}
&\ric^{N+d,\m^N_{\scriptscriptstyle{C}}}_{_{\scriptscriptstyle{C}}}(\xi+v)=\ric_{{B}}(\xi)-\textstyle{N\frac{\nabla^2f(\xi)}{f(p)}}\\
&\hspace{4cm}+\ric^{N,\m_{\scriptscriptstyle{F}}}_{\scriptscriptstyle{F}}(v)-\left[\textstyle{\frac{\Delta^Bf(p)}{f(p)}}+(N-1)\textstyle{\frac{|\nabla f_p|^2}{f^2(p)}}\right]|v|_{\sC}^2
\end{align*}
and the previous assumptions together with
$\ric^{N,\m_{\scriptscriptstyle{F}}}_{{\scriptscriptstyle{F}}}(v,v)\geq(N-1)K_{\scriptscriptstyle{F}}|v|_{\sF}^2$ imply
\begin{equation*}\ric^{N+d,\m^N_{\scriptscriptstyle{C}}}_{\sC}(\xi+v,\xi+v)\geq (n+d-1)K|\xi+v|_{\sC}^2\,
\end{equation*}
for every $v\in T F$ and for every $\xi+v\in TB\times_f F=TB\oplus TF$ respectively. 
\medskip\\
Let $(F,g_{\sF})=\F$ be a Riemannian manifold and $Z$ a smooth vector field on $F$ such that $(\F,Z)$ satisfies the condition $CD(K_{\sF}(N-1),N)$. By 
Theorem \ref{maintheorem} the condition is equivalent to $\ric_{\sF,\sZ}^{\sN}\geq K_{\sF}(N-1)$. 

\begin{theorem}
Let $f$, $(F,g_{\sF},\vol_{\sF})=\F$, and $Z$ be as before, with constants $K,K_{\sF}\in \mathbb{R}$ and $N\geq 1$. 
Assume $B$ has Alexandrov curvature bounded from below by $K$, and $f^{-1}(\left\{0\right\})\subset \partial B$. If $N=1$ and $K_{\sF}>0$, we assume $\diam_{\sF}\leq \pi/\sqrt{K_{\sF}}$. 
\smallskip\\
Then the $N$-warped product $B\times_{f}^{\sN}F$ is a smooth metric measure space, and $(B\times_{f}^{\sN}F, Z^{\flat})$ satisfies the condition $CD(K(N+d-1),N+d)$ where $Z^{\flat}={f^{-2}}Z$.
\end{theorem}
\begin{proof}
First,
by similar computations as in the proof of Proposition 3.2 in \cite{ketterer} we obtain that the $N+d$-Ricci tensor of $L^C$ is given by
\begin{align*}
\ric^{\sN+d}_{\scriptscriptstyle{C},\sZ^{\flat}}(\xi+v)=&\ric_{\scriptscriptstyle{B}}(\xi)-N\frac{\nabla^2f(\xi)}{f(p)}+\ric^{\sN}_{\scriptscriptstyle{F},\sZ}(v)-\left[\textstyle{\frac{\Delta^Bf(p)}{f(p)}}+(N-1)\textstyle{\frac{|\nabla f_p|^2}{f^2(p)}}\right]|v|_{\scriptscriptstyle{C}}^2
\end{align*}
and under the previous assumptions together with
$\ric^{\sN}_{\scriptscriptstyle{F},\sZ}(v)\geq(N-1)K_F|v|_{\scriptscriptstyle{F}}^2$, this implies
\begin{equation*}
\ric^{N+d,}_{\scriptscriptstyle{C},\sZ^{\flat}}(\xi+v)\geq (N+d-1)K|\xi+v|_{\scriptscriptstyle{C}}^2.
\end{equation*}
\smallskip
\noindent
Moreover, if $K_{\sF}>0$, then by the generalized Bonnet-Myers theorem we know that $\diam_{\sF}\leq \pi{\scriptstyle \sqrt{\frac{1}{K_{\sF}}}}$.
Note, by Theorem 1.2 in \cite{albi} condition (iii) from the beginning - provided condition (ii) - is equivalent to 
\begin{itemize}
\smallskip
 \item[(1)] $K_{\sF}\geq f^2 K$ if $f^{-1}(\left\{0\right\})=\emptyset$,
 \medskip
 \item[(2)] $K_{\sF}>0$ and $|\nabla f|\leq \sqrt{K_{\sF}}$ if $f^{-1}(\left\{0\right\})\neq\emptyset$.
 \medskip
\end{itemize}
Then, exactly like in the proof of Theorem 3.4 in \cite{ketterer} we can prove
if $\Pi$ is an optimal dynamical transference plan in $C$
such that $(e_0)_{\star}\Pi$
is 
absolutely continuous with respect to $\m^{\sN}_f$, then
\begin{align*}
\Pi\left(\left\{\gamma\in\mathcal{G}(X): \exists\ t\in (0,1) \mbox{ such that } \gamma(t)\in f^{-1}(\left\{0\right\})\right\}\right)=0.
\end{align*}
Hence, the $N$-warped product is indeed a generalized smooth metric measure space in the sense of our definition. For the proof we need a maximal diameter theorem 
for $(F,g_{\sF},\vol_{\sF},Z)$ that satisfies the condition $CD(K_{\sF},N)$ with $K_{\sF}>0$. For smooth $Z$ this is provided by Kuwada in \cite{kuwadamaximaldiameter} (We also refer to \cite{limoncu} for closely related results). 
\smallskip\\
Finally, following the lines of the proof of the main theorem in \cite{ketterer} in combination with Theorem \ref{maintheorem}, we obtain the condition $CD(K,N)$ with admissible vector field $Z^{\flat}$.
\end{proof}
\begin{remark}
\begin{itemize}
 \item[(i)] In particular, the previous theorem covers the case of $N$-cones and $N$-suspensions that play an important role for the study of spaces with generalized lower Ricci curvature bounds. 
 \item[(ii)] $\tilde{Z}$ is $L^2$-integrable w.r.t. $\m^{\sN}_f$ if $N>1$.
\end{itemize}
\end{remark}

\begin{example}
Let $B=[0,\pi]$ and $F=\mathbb{S}^2$, and let $Z$ be any smooth vector field on $\mathbb{S}^2$. 
Since $Z$ is smooth and $\mathbb{S}^2$ is compact, we can find a constant $\kappa>0$ such that 
\begin{align*}
-\nabla^sZ(v,v)-\langle Z,v\rangle^2\geq -\kappa|v|^2\ \ \mbox{ for any }v\in T\mathbb{S}^2.
\end{align*}
Now, we consider $\alpha Z$ with $\alpha>0$ such that $\alpha \kappa \leq \frac{1}{2}$ and choose $N>2$ such that $\frac{1}{N-2}\leq \frac{1}{\alpha}$. 
Hence
\begin{align*}
-\nabla^s\alpha Z(v,v)-\frac{1}{N-2}\langle \alpha Z,v\rangle^2= -\alpha\left[\nabla^sZ(v,v)+\frac{\alpha}{N-2}\langle Z,v\rangle^2\right]\geq -\alpha \kappa|v|^2.
\end{align*}
and therefore
\begin{align*}
\ric_{\sF,\alpha Z}^{\sN}(v,v)&=\ric_{\sF}(v,v)-\nabla^s\alpha Z(v,v)-\frac{1}{N-2}\langle \alpha Z,v\rangle^2\\
&\geq (1-\alpha\kappa)|v|^2\geq \frac{1}{2}|v|^2=\frac{1}{2(N-2)}(N-2)|v|^2.
\end{align*}
Now, we set $K_{\sF}:= \frac{1}{2(N-2)}$ and $f:[0,\pi]\rightarrow [0,\infty)$ with $f(r)=\sqrt{K_{\sF}}\sin(r)$.
The previous theorem yields that $B\times_{f}^{\sN}F$ with $\frac{1}{\sin^2}\alpha Z=\alpha\tilde{Z}$ satisfies the condition $CD(N,N+1)$.
Note $\ric_{\sF,Z}^{\sN}(v,v)=g_{\sF}(v,v)$ one only can achieve if we set $\alpha=0$. Then, any $N\geq 2$ is admissible, $Z=0$ and - for $N=2$ - $B\times_f^{2}F\simeq\mathbb{S}^3$. 
Otherwise, the underlying metric space has singularities
at the vanishing points of $f$.
Also note that $f^{\sN}rdr\otimes d\vol_{\mathbb{S}^2_{}}$ is an invariant measure of the corresponding 
nonsymmetric diffusion operator
$$\frac{d^2}{dr^2}+\frac{N}{\sin}\frac{d}{dr}+\frac{1}{\sin^2}\left(\Delta_{\mathbb{S}^2_{\beta}}+\alpha Z\right)=\Delta_{B\times_{f}^{\sN}F}+\alpha\tilde{Z}$$
on $B\times_{f}^{\sN}F$ where $\tilde{Z}=\frac{1}{\sin^2}Z$. 
We observe that the diameter bound $\pi$ of the generalized Bonnet-Myers theorem is attained but the $N$-warped product is not a smooth manifold unless $\alpha=0$ and $N=2$! 
In \cite{kuwadamaximaldiameter} Kuwada proves that for $(\X,Z)$ (with $X=M$ a smooth manifold, $\m_{\sX}=\vol_{\sM}$, and $Z$ a smooth vector field) 
that satisfies $CD(n-1,n)$ with $n\in \mathbb{N}$ and $n\geq 2$ the maximal diameter is attained if and only if $X=\mathbb{S}^n$ and $Z=0$.
We also observe that the generalized Laplace operator splits into a nonsymmetric part $\frac{\alpha}{\sin^2}Z$ and a symmetric part $\Delta_{B\times_{f}^{\sN}F}$
where the invariant measure $f^{\sN}rdr\otimes d\vol_{\mathbb{S}^2_{\beta}}$ is also the invariant measure for the symmetric part, and $\int Zg f^{\sN}rdr\otimes d\vol_{\mathbb{S}^2_{\beta}}=0$ for any smooth function $g$.
\end{example}

\section{Evolution variational inequality and Wasserstein control}
\noindent
In this section we discuss the link between the curvature-dimension condition and contraction estimates in Wasserstein space. 
Let $(X,\de_{\sX},\m_{\sX})$ be a compact generalized smooth metric measure space that satisfies the condition $CD(K,N)$ for an admissible vector field $Z$. 
We assume that $X=M_{}$, $\m_{\sX}=\vol_{\sM}$ and $Z$ is smooth. Hence, $\de_{\sX}$ is induced by a smooth Riemannian metric $g_{\sM}$ on $M$.
Theorem \ref{maintheorem} yields
\begin{align*}
\ric_{\sM,\sZ}^{\sN}\geq Kg_{\sM}
\end{align*}
for $K\in\mathbb{R}$ and $N\in [1,\infty]$. We also assume $N=\infty$. \\
\\
We can consider the diffusion operator ${L}=\Delta + Z$ and the corresponding nonsymmetric Dirichlet form
\begin{align*}
\mathcal{E}(u)=\int|\nabla u|^2d\vol_{\sM} + \int Z(u)u d\vol_{\sM}
\end{align*}
on $L^2(\vol_{\sM})$
where $D(\mathcal{E})=W^{1,2}(X)$. Note that $\mathcal{E}$ is a  Dirichlet form adapted to the Dirichlet energy w.r.t. $g_{\sM}$ in the sense of \cite{lierlsaloffcoste}, and
$\Delta$ is the Laplace-Beltrami operator of $(M,g_{\sM})$. Let $P_t$ be the associated diffusion semigroup, let $P_t^*$ be the so-called co-semigroup, 
and let $\mathcal{H}_t$  be the dual flow acting on probability measures:
\begin{align*}
\int  (P_tu) v d\vol_{\sM}= \int u P_t^*vd\vol_{\sM}, \ \ \int \phi d\mathcal{H}_t\mu=\int P_t\phi d\mu
\end{align*}
where $u,v\in L^2(\vol_{\sM})$ and $\phi\in C_b(X)$. $L^*$ denotes the generator of $P_t^*$. Note
\begin{align}\label{uuuu}
\frac{d}{dt}\int u P_t^*v d\vol_{\sM} = -\int u L^* vd\vol_{\sM}=\mathcal{E}(u,v).
\end{align}
For more details about nonsymmetric Dirichlet forms and the relation between $P_t$, $P^*_t$, $L$, $L^*$ and $\mathcal{E}$ we refer to \cite{maroeckner}.
Since
$
\int d\mathcal{H}_t\mu=\int P_t1d\mu=1,
$
we have $\mathcal{H}_t:\mathcal{P}^2(X)\rightarrow \mathcal{P}^2(X)$.\smallskip\\
We assume that there is a $C^{\infty}$-kernel $p_t(x,y)$ for $P_t$, i.e. 
$$
P_t u(x)=\int_X p_t(x,y)u(y)d\vol_{\sM}(y)
$$ 
such that $p:(0,\infty)\times X^2\rightarrow (0,\infty)$ is $C^{\infty}$.
Then 
\begin{align*}
\int \phi(x) d\mathcal{H}_t\mu(x)=\int_X \int_X p_t(x,y)\phi(y)d\vol_{\sM}(y) d\mu(x)
\end{align*}$\mbox{for any } \phi\in C_b(X).$
In particular, $\mathcal{H}_t\mu$ is $\vol_{\sM}$-absolutely continuous for any $\mu\in\mathcal{P}^2(M)$, and its density $\rho_t(y)=\int p_t(x,y) d\mu(x)$ is smooth. 
Note that we can extend $P_t^*$ (and $P_t$ as well) as semigroup acting on $L^1(\vol_{\sM})$. 
Then, for $\vol_{\sM}$-absolutely continuous probability
measures $\mu=\rho\vol_{\sM}$ we have $\mathcal{H}_t\mu=P^*_t \rho d\vol_{\sM}$.
By abuse of notation we denote with $L$ the $L^1$-generator of $P_t$, and similar for $P_t^*$. Then, for instance (\ref{uuuu}) extends in a canonical way provided $L^{\infty}$-integrability for $u$. 
This is straightforward, and we omit details. 

\begin{proposition}[Kuwada's Lemma]\label{kuwadaslemma} Let $(M,g_{\sM}, \vol_{\sM})$ and $Z$ be as before. Let $\mu\in\mathcal{P}^2(M)$.
Then $t\in (0,\infty)\mapsto \mathcal{H}_t\mu$ is a locally absolutely continuous curve in $\mathcal{P}^2(\vol_{\sM})$, and 
\begin{align}\label{metricspeed}
|\dot{\mathcal{H}}_{t}\mu|^2=\lim_{s\rightarrow 0}\frac{1}{s^2}W_2(\mathcal{H}_t\mu,\mathcal{H}_{t+s}\mu)^2\leq \int_M  |\nabla\log\rho_{t}-Z|^2d\mathcal{H}_{t}\mu .
\end{align}
In particular, it is differentiable almost everywhere in $t$.
\end{proposition}
\begin{proof}
Since $\mathcal{H}_t\mu\in\mathcal{P}^2(\vol_{\sM})$ for $t>0$ and $\mu\in\mathcal{P}^2(M)$, it is enough to consider $\mu=\rho\vol_{\sM}\in\mathcal{P}^2(\vol_{\sM})$.
Fix $t,s>0$. By Kantorovich duality there exists a Lipschitz function $\varphi$ such that
\begin{align*}
\frac{1}{2}W_2(\mathcal{H}_t\mu,\mathcal{H}_{t+s}\mu)^2=\int Q_1\varphi d\mathcal{H}_{t+s}\mu-\int\varphi d\mathcal{H}_t\mu.
\end{align*} 
where $r\in[0,1]\rightarrow Q_r$ is the Hopf-Lax semigroup. $r\mapsto Q_r\varphi\in L^2(\vol_{\sM})$ is differentiable and satisfies the Hamilton-Jacobi equation
\begin{align*}
\frac{d}{dr}Q_r\varphi+\frac{1}{2}|\nabla Q_r\varphi|^2=0
\end{align*}
in the sense of \cite{agsheat}.
Since $\mathcal{H}_t\mu=\int p_t(x,\cdot)d\mu(x)=P_t^*\rho\vol_{\sM}$ the curve $t\in (0,\infty)\mapsto P_t^*\rho\in L^2(\vol_{\sM})$ is locally Lipschitz, and therefore $r\in [0,1]\mapsto Q_r\varphi P_{t+rs}^*\rho$ is Lipschitz as well. 
Therefore
\begin{align*}
&\int Q_1\varphi d\mathcal{H}_{t+s}\mu-\int\varphi d\mathcal{H}_t\mu=\int_0^1\frac{d}{dr}\int Q_r\varphi P_{t+rs}^*\rho dr \\
&=-\int_0^1\int\frac{1}{2}|\nabla Q_r\varphi|^2d\mathcal{H}_{t+rs}\mu dr+\int_0^1\int (Q_r\varphi) \frac{d}{dr}P_{t+rs}^*\rho d\vol_{\sM} dr.
\end{align*}
Furthermore, we have
\begin{align*}
\int (Q_r\varphi) \frac{d}{dr}P_{t+rs}^*\rho d\vol_{\sM}&=s\int (Q_r\varphi )L^*P_{t+rs}^*\rho d\vol_{\sM}\\
&=-s\int \langle \nabla Q_r\varphi,\nabla \rho_{t+rs}\rangle d\vol_{\sM}+s\int ZQ_r\varphi P_{t+rs}^*\rho d\vol_{\sM}\\
&=-s\int \langle \nabla Q_r\varphi ,\frac{\nabla \rho_{t+rs}}{\rho_{t+rs}}-Z\rangle \rho_{t+rs}d\vol_{\sM}.
\end{align*}
Now, recall that the inequality
$
-\langle\nabla g,\nabla \tilde{g}\rangle\leq \frac{1}{2s}|\nabla g|^2+\frac{s}{2}|\nabla\tilde{g}|^2.
$ Hence
\begin{align*}
\int (Q_r\varphi) \frac{d}{dr}P_{t+rs}^*\rho d\vol_{\sM}&\leq \frac{1}{2}\int|\nabla Q_r\varphi|^2d\mathcal{H}_{t+rs}\mu+\frac{s^2}{2}\int |\nabla\log\rho_{t+rs}-Z|^2d\mathcal{H}_{t+rs}\mu
\end{align*}
and consequently
\begin{align}\label{hmh}
\int Q_1\varphi d\mathcal{H}_{t+s}\mu-\int\varphi d\mathcal{H}_t\mu&\leq \frac{s^2}{2}\int_0^1\int |\nabla\log\rho_{t+rs}-Z|^2d\mathcal{H}_{t+rs}\mu dr.
\end{align}
By smoothness of the heat kernel $p_t$, $t\mapsto \mathcal{H}_{t+rs}\mu\in\mathcal{P}^2(M)$ is locally Lipschitz. Moreover, (\ref{metricspeed}) easily follows from (\ref{hmh}).
\end{proof}
\begin{proposition}\label{propA}
Let $(M,g_{\sM},\vol_{\sM})$, $Z$ and $\mathcal{H}_t\mu$ be as above. Then 
\begin{align*}
\frac{1}{2}\frac{d}{ds}W^2_2(\mathcal{H}_s\mu,\nu)+\frac{K}{2}W^2_2(\mathcal{H}_s\mu,\nu)\leq \int_0^1\int \alpha(\dot{\gamma})d\Pi^s(\gamma)dt + \Ent(\nu)-\Ent(\mathcal{H}_s\mu).
\end{align*}
where $\nu\in \mathcal{P}^2(M)$, $\Pi^t$ is the dynamical optimal plan between $\mathcal{H}_t\mu$ and $\nu$, and $\alpha$ is the co-vectorfield that corresponds to $Z$.
\end{proposition}
\begin{proof} We set $\mathcal{H}_t\mu=\mu_t=\rho_t\vol_{\sM}$, and let $(\mu_{t,r})_{r\in[0,1]}$ be the $L^2$-Wasserstein geodesic between $\mu_t=\mu_{t,0}$  and $\nu=\mu_{t,1}\in \mathcal{P}^2(M)$. 
Let $\Pi^t\in \mathcal{P}(\mathcal{G}(M))$ be the corresponding optimal dynamical plan. By the McCann-Brenier theorem \cite{mccann} we know that $\Pi^t$ and $(\mu_{t,r})_{r\in[0,1]}$ are unique and 
$\mu_{t,r}=\rho_{t,r}\vol_{\sM}\in\mathcal{P}^2(\vol_{\sM})$ for $r\in [0,1)$. By the semigroup property of $p_t$ it is clear that $\mathcal{H}_t\mu=P_{t/2}^*\rho_{t/2}\vol_{\sM}$. 
Therefore, without restriction we can assume $\mu\in\mathcal{P}^2(\vol_{\sM})$.
Consider
\begin{align*}
\Ent(\mu_{\tau})-\int \int_0^{\tau}\langle Z|_{\gamma(r)},\dot{\gamma}(r)\rangle dr d\Pi^t=:\mathcal{S}_{\tau}(\Pi).
\end{align*}
Let $\varphi^t<\infty$ be a Kantorovich potential for the geodesic $(\mu_{t,r})_{r\in [0,1]}$. Recall that a Kantorovich potential $\varphi^t<\infty$ on $X$ is Lipschitz.
Moreover, $\rho_{t,0}\in C^{\infty}(X)$ since $p_t$ is smooth.
We compute exactly as in the proof of Theorem 6.5 of \cite{agmr}
\begin{align}\label{dada}
\frac{d}{d\tau}\Big|_{\tau=0}\mathcal{S}_{\tau}(\Pi^t)&=\frac{d}{d\tau}\Big|_{\tau=0}\Ent(\mu_{t,\tau})-\frac{d}{d\tau}\Big|_{\tau=0}\int \int_0^{\tau}\langle Z|_{\gamma_r}, \dot{\gamma}_r\rangle dr d\Pi^t\nonumber\\
&\geq - \int\langle\nabla \varphi^t,\nabla \rho_{t,0}\rangle d\vol_{\sM}+\int \langle Z,\nabla\varphi^t\rangle \rho_{t,0}d\vol_{\sM}.
\end{align}
Note again that $\gamma_t=-\nabla\phi(\gamma_0)$ for $\Pi$-a.e. $\gamma\in\mathcal{G}(M)$.
On the other hand, we can compute the derivative of the $L^2$-Wasserstein distance along $\mathcal{H}_t\mu$. First, by Kantorovich duality we have
\begin{align*}
\frac{1}{2}W_2(\mathcal{H}_t\mu,\nu)^2=\int \varphi^t d\mathcal{H}_t\mu+\int Q_1(-\varphi^t) d\nu
\end{align*}
and 
\begin{align*}
\frac{1}{2}W_2(\mathcal{H}_{t-h}\mu,\nu)^2\geq \int \varphi^t d\mathcal{H}_{t-h}\mu+\int Q_1(-\varphi^t) d\nu
\end{align*}
Note, that we have $Q_1(-\varphi)=\varphi^c$ (we changed the sign convention for this proof).
Then, for $t>0$ and $h>0$ we have
\begin{align*}
&\liminf_{h\downarrow 0}\frac{1}{h}\left(W_2(\mathcal{H}_{t-h}\mu,\nu)^2-W_2(\mathcal{H}_t\mu,\nu)^2\right)\\
&\geq\liminf_{h\downarrow 0}\frac{1}{h}\left[\int \varphi^t d\mathcal{H}_{t-h}\mu-\int \varphi^t d\mathcal{H}_t\mu\right]\\
&=\lim_{h\downarrow 0}\frac{1}{h}\left[\int \varphi^t P^*_{t-h}\rho_{0}d\vol_{\sM}-\int \varphi^t P^*_{t}\rho_{0}d\vol_{\sM}\right]\\
&=-\int \varphi^t\frac{d}{dt}P^*_{t}\rho_{0} d\vol_{\sM}=-\int \varphi^tL^*P_t^*\rho_{0} d\vol_{\sM}\\
&=\int\langle\nabla\varphi^t,\nabla\rho_{t,0} \rangle d\vol_{\sM}-\int\langle Z|_{x},\nabla\varphi^t(x)\rangle \rho_{t,0}(x)d\vol_{\sM}.
\end{align*}
The last equality follows from the relation between $\mathcal{E}$ and its dual generator $L^*$ \cite{maroeckner}.
Therefore $$\frac{d^-}{ds}\Big|_{s=t}W_2(\mathcal{H}_t\mu,\nu)^2\leq -\int\langle\nabla\varphi^t,\nabla\rho_{t,0} \rangle d\vol_{\sM}+\int\langle Z|_{x},\nabla\varphi^t(x)\rangle \rho_{t,0}(x)d\vol_{\sM},$$
where $\frac{d^-}{ds}\Big|_{s=t}f(t)=\limsup\limits_{h\uparrow 0}\frac{1}{h}[f(t+h)-f(t)]$. Together with (\ref{dada}) and since $t\mapsto \mathcal{H}_t\mu$ is absolutely continuous, it follows
\begin{align}\label{pretty}
\frac{d}{d\tau}\Big|_{\tau=0}\mathcal{S}_{\tau}(\Pi)&\geq \frac{d}{ds}\Big|_{s=t}W_2(\mathcal{H}_s\mu,\nu)^2
\end{align}
The curvature-dimension condition implies that for $\mathcal{H}_t\mu$, $\nu$ and $\Pi$ with $(e_r)_{\star}\Pi=\mu_r$ 
\begin{align*}
\Ent(\mu_t)-\phi_{t}(\Pi)\leq (1-t)\Ent(\mu_0)&+t\left[\Ent(\mu_1)-\phi_{1}(\Pi)\right]
\\
&\hspace{1cm}
-\frac{1}{2}Kt(1-t)KW_2(\mu_0,\mu_1)^2,
\end{align*}
where $\phi_{t}(\Pi)=\int \phi_t(\gamma)d\Pi(\gamma)$. Now we substract 
$\Ent(\mu_0)$, devide by $t>0$, and let $t\rightarrow 0$. We obtain
\begin{align*}
\frac{d}{d\tau}\Big|_0 S_{\tau}(\Pi)\leq -\Ent(\mathcal{H}_t\mu)+\left[\Ent(\nu)-\phi_{1}(\Pi)\right]-\frac{1}{2}K W_2(\mathcal{H}_t\mu,\nu)^2
\end{align*}
Together with (\ref{pretty}) this implies the result.
\end{proof}
\begin{definition}
Let $\mms$ be a metric measure space, and let $Z$ be an $L^2$-integrable vector field.
Let $K\in\mathbb{R}$ and $(\mu_t)_{t\in (0,\infty)}\subset \mathcal{P}^2(\vol_{\sM})$ be locally absolutely continuous. We say $\mu_t$ is an $EVI_{K,\infty}$-flow curve of $Z$ starting in $\mu_0\in\mathcal{P}^2(X)$ if 
$\lim_{t\rightarrow 0}\mu_t=\mu_0$, and 
for every $\nu\in \mathcal{P}^2(\vol_{\sM})$ the \textit{evolution variational inequality} 
\begin{align*}
\frac{1}{2}\frac{d}{ds}W^2_2(\mathcal{H}_s\mu,\nu)+\frac{K}{2}W^2_2(\mathcal{H}_s\mu,\nu)\leq -\int_0^1\int Z(\dot{\gamma})d\Pi^s(\gamma)dr + \Ent(\nu)-\Ent(\mathcal{H}_s(\mu)).
\end{align*}
holds for a.e. $t>0$ where $\Pi^s$ is the dynamical optimal plan between $\mathcal{H}_s\mu$ and $\nu$.
\end{definition}
\begin{corollary} 
Assume $(M,g_{\sM},\vol_{\sM},Z)$ as above satisfies the condition $CD(K,N)$. Then
$t\mapsto \mathcal{H}_t\mu$ is an $EVI_{K,\infty}$ gradient flow curve for every $\mu\in\mathcal{P}^2(M)$.
\end{corollary}
\noindent
\noindent
The following contraction estimate for nonsymmetric diffusions was known before.
\begin{corollary}\label{contractionxxx}
Let $(M,g_{\sM},\vol_{\sM})$ and $Z$ be as before. Then 
\begin{align*}
W_2(\mathcal{H}_t\mu,\mathcal{H}_t\nu)^2\leq e^{-2Kt}W_2(\mu,\nu)^2.
\end{align*}
for $\mu,\nu\in \mathcal{P}^2(M)$.
\end{corollary}
\begin{proof} First, note that by the previous corollary and Lemma 4.3.4 in \cite{agsgradient} the function $t\in (0,\infty)\mapsto W_2(\mathcal{H}_t\mu,\mathcal{H}_t\nu)^2$ is 
locally absolutely continuous, and therefore differentiable almost everywhere. 
\smallskip\\
Let $t<s$.
Integrating the $EVI_K$ inequality from $t$ to $s$ with $\nu=\mathcal{H}_s\mu$ yields
\begin{align*}
&\frac{1}{2}W^2_2(\mathcal{H}_s\mu,\mathcal{H}_s\nu)-\frac{1}{2}W^2_2(\mathcal{H}_t\mu,\mathcal{H}_s\nu)+\int_t^s\frac{K}{2}W^2_2(\mathcal{H}_{\tau}\mu,\mathcal{H}_s\nu)d\tau\\
&\leq -\int_t^s\int_0^1\int Z(\dot{\gamma})d\Pi^{\tau}(\gamma)dr d\tau + (s-t)\Ent(\mathcal{H}_s\nu)-\int_t^s\Ent(\mathcal{H}_{\tau}(\mu))d\tau.
\end{align*}
where $\Pi^{\tau}$ is the geodesic between $\mathcal{H}_{\tau}\mu$ for $\tau\in [t,s]$ and $\mathcal{H}_s\nu$,
and similar if we set $\nu=\mathcal{H}_t\mu$, we obtain
\begin{align*}
&\frac{1}{2}W^2_2(\mathcal{H}_s\mu,\mathcal{H}_t\nu)-\frac{1}{2}W^2_2(\mathcal{H}_t\mu,\mathcal{H}_t\nu)+\int_t^s\frac{K}{2}W^2_2(\mathcal{H}_{t}\mu,\mathcal{H}_{\tau}\nu)d\tau\\
&\leq -\int_t^s\int_0^1\int Z(\dot{\gamma})d\Pi^{\tau,-}(\gamma)dr d\tau + (s-t)\Ent(\mathcal{H}_t\mu)-\int_t^s\Ent(\mathcal{H}_{\tau}(\nu))d\tau.
\end{align*}
where $\Pi^{\tau,-}$ is the geodesic between $\mathcal{H}_{\tau}\nu$ for $\tau\in [t,s]$ and $\mathcal{H}_t\mu$. Adding the previous inequalities from each other yields
\begin{align*}
&\frac{1}{2}W^2_2(\mathcal{H}_s\mu,\mathcal{H}_s\nu)-\frac{1}{2}W^2_2(\mathcal{H}_t\mu,\mathcal{H}_t\nu)+\frac{K}{2}\int_t^s\left[W^2_2(\mathcal{H}_{t}\mu,\mathcal{H}_{\tau}\nu)+W^2_2(\mathcal{H}_{\tau}\mu,\mathcal{H}_s\nu)\right]d\tau\\
&\leq -\int_t^s\int_0^1\int Z(\dot{\gamma})d\Pi^{\tau}(\gamma)dt d\tau-\int_t^s\int_0^1\int Z(\dot{\gamma})d\Pi^{\tau,-}(\gamma)dt d\tau\\
&\hspace{2cm}+(s-t)(\Ent(\mathcal{H}_s\mu)+\Ent(\mathcal{H}_t\mu))-2\int_t^s\Ent(\mathcal{H}_{\tau}(\nu))d\tau.
\end{align*}
Deviding by $(s-t)$, and letting $s\rightarrow t$ yields
\begin{align}\label{seethelastline}
\frac{d}{d\tau}\Big|_{\tau=t}\frac{1}{2}W^2_2(\mathcal{H}_{\tau}\mu,\mathcal{H}_{\tau}\nu)&\leq - K W^2_2(\mathcal{H}_{t}\mu,\mathcal{H}_{t}\nu)\nonumber\\
&-\int_0^1\left(\int Z(\dot{\gamma})d\Pi^{t}(\gamma)+\int Z(\dot{\gamma})d\Pi^{t,-}(\gamma)\right)dt.
\end{align}
Since geodesic between absolutely continuous probability measures are unique, and since $\Pi^{\tau,-}$ coincides with $\Pi^{\tau}$ up to reverse parametrization, we have for $\Psi(\gamma)=\gamma^-$
$$\int Z(\dot{\gamma})d\Pi^{t}(\gamma)=\int Z(\dot{\gamma}^-)d\Psi_{\star}\Pi^{t}(\gamma)=-\int Z(\dot{\gamma})d\Pi^{t,-}(\gamma)$$ and the last line in (\ref{seethelastline}) vanishes. 
Hence
\begin{align*}
\frac{d}{dt}\Big|_{t=0}W_2(\mathcal{H}_t\mu,\mathcal{H}_t\nu)^2\leq - 2KW_2(\mathcal{H}_t\mu,\mathcal{H}_t\nu)^2.
\end{align*}
Finally, Gromwall's lemma yields the claim.
\end{proof}

\begin{corollary}[Kuwada, \cite{kuwadaduality}]
For $f\in D(\mathcal{E})$, we have
\begin{align*}
\left|\nabla P_t f\right|^2\leq e^{-2K}P_t\left|\nabla f\right|^2.
\end{align*}
Moreover, $P_t$ satisfies the Bakry-Emery curvature-dimension condition $BE(K,\infty)$.
\end{corollary}
\begin{remark}
By a classical result of Bakry and Emery (for instance \cite{saintfleur}) the condition $BE(K,\infty)$ again implies $\ric_{\sM,\infty}^{\sZ}\geq K$ and therefore the condition $CD(K,\infty)$.
\end{remark}
\begin{remark}
It is rather obvious how to improve the previous estimates in context of the condition $CD(K,N)$ for $N<\infty$ in the sense of \cite{erbarkuwadasturm} by modifying the computations (see also \cite{kuwadaspacetime}), 
and we omit details. 
\end{remark}
\noindent
We summarize the previous statements in the following theorem. 
\begin{theorem}
Let $(M,g_{\sM})$ be a compact smooth Riemannian manifold, and let $Z$ be an smooth vector field. We denote with $(M,\de_{\sM},\vol_{\sM})$ the corresponding metric measure space, 
and let $P_t$ and $\mathcal{H}_t$ be as above.
Then, the following statements are equivalent.
\begin{itemize}
\smallskip
 \item[(i)]$\ric_{\sM,\infty}^{\sZ}\geq K$, 
\smallskip
 \item[(ii)]$(M,\de_{\sM},\vol_{\sM},Z)$ satisfies the condition $CD(K,\infty)$,
\smallskip
 \item[(iii)]For every $\mu\in\mathcal{P}^2(X)$ $\mathcal{H}_t\mu$ is an $EVI_{K,\infty}$-flow curve starting in $\mu$,
\smallskip
 \item[(iv)]$\mathcal{H}_t$ satisfies the contraction estimate in corollary \ref{contractionxxx},
\smallskip
 \item[(v)]$P_t$ satisfies the condition $BE(K,\infty)$.
\end{itemize}
\end{theorem}

\noindent
\small{
\bibliography{new}

\bibliographystyle{amsalpha}}

\end{document}